\documentclass{elsarticle}

\usepackage{tikz}
\usepackage[caption=false,font=footnotesize]{subfig}

\usepackage{makecell, multirow, threeparttable, tabularx, array, arydshln}
\newcolumntype{Z}{>{\centering\let\newline\\\arraybackslash\hspace{0pt}}X}
\usepackage{threeparttable}

\usepackage{amsmath}
\usepackage{amsthm}
\usepackage{amssymb, amsfonts, mathtools}
\usepackage{bm}

\DeclareSymbolFont{cmsymbols}{OMS}{cmsy}{m}{n}
\SetSymbolFont{cmsymbols}{bold}{OMS}{cmsy}{b}{n}
\DeclareSymbolFontAlphabet{\mathcal}{cmsymbols}
\newcommand{\mat}[1]{\bm{#1}}
\newcommand{\ten}[1]{\bm{\mathcal{#1}}}
\newcommand{\core}[2]{\bm{\mathcal{#1}}^{(#2)}}

\newcommand{\argmin}[1]{\underset{#1}{\operatorname{argmin}}}

\newcommand{\inR}[1]{\in\mathbb{R}^{#1}}
\DeclareMathOperator{\vect}{vec}
\DeclareMathOperator{\reshape}{reshape}
\DeclareMathOperator{\spl}{unfold}

\DeclareMathOperator{\svd}{SVD}
\DeclareMathOperator{\rank}{rank}

\makeatletter
\let\save@mathaccent\mathaccent
\newcommand*\if@single[3]{%
    \setbox0\hbox{${\mathaccent"0362{#1}}^H$}%
    \setbox2\hbox{${\mathaccent"0362{\kern0pt#1}}^H$}%
    \ifdim\ht0=\ht2 #3\else #2\fi
}
\newcommand*\rel@kern[1]{\kern#1\dimexpr\macc@kerna}
\newcommand*\widebar[1]{\@ifnextchar^{{\wide@bar{#1}{0}}}{\wide@bar{#1}{1}}}
\newcommand*\wide@bar[2]{\if@single{#1}{\wide@bar@{#1}{#2}{1}}{\wide@bar@{#1}{#2}{2}}}
\newcommand*\wide@bar@[3]{%
    \begingroup
    \def\mathaccent##1##2{%
        \let\mathaccent\save@mathaccent
        \if#32 \let\macc@nucleus\first@char \fi
        \setbox\z@\hbox{$\macc@style{\macc@nucleus}_{}$}%
        \setbox\tw@\hbox{$\macc@style{\macc@nucleus}{}_{}$}%
        \dimen@\wd\tw@
        \advance\dimen@-\wd\z@
        \divide\dimen@ 3
        \@tempdima\wd\tw@
        \advance\@tempdima-\scriptspace
        \divide\@tempdima 10
        \advance\dimen@-\@tempdima
        \ifdim\dimen@>\z@ \dimen@0pt\fi
        \rel@kern{0.6}\kern-\dimen@
        \if#31
        \overline{\rel@kern{-0.6}\kern\dimen@\macc@nucleus\rel@kern{0.4}\kern\dimen@}%
        \advance\dimen@0.4\dimexpr\macc@kerna
        \let\final@kern#2%
        \ifdim\dimen@<\z@ \let\final@kern1\fi
        \if\final@kern1 \kern-\dimen@\fi
        \else
        \overline{\rel@kern{-0.6}\kern\dimen@#1}%
        \fi
    }%
    \macc@depth\@ne
    \let\math@bgroup\@empty \let\math@egroup\macc@set@skewchar
    \mathsurround\z@ \frozen@everymath{\mathgroup\macc@group\relax}%
    \macc@set@skewchar\relax
    \let\mathaccentV\macc@nested@a
    \if#31
    \macc@nested@a\relax111{#1}%
    \else
    \def\gobble@till@marker##1\endmarker{}%
    \futurelet\first@char\gobble@till@marker#1\endmarker
    \ifcat\noexpand\first@char A\else
    \def\first@char{}%
    \fi
    \macc@nested@a\relax111{\first@char}%
    \fi
    \endgroup
}
\makeatother

\usepackage{algorithm}
\usepackage{algorithmic}
\usepackage{eqparbox}

\usepackage{etoolbox}  
\makeatletter
\patchcmd{\algorithmic}{\addtolength{\ALC@tlm}{\leftmargin}}{\addtolength{\ALC@tlm}{\leftmargin}}{}{}
\makeatother

\usepackage[hidelinks]{hyperref} 
\usepackage{booktabs} 
\usepackage{cleveref} 
\newtheorem{theorem}{Theorem}
\newtheorem{lemma}[theorem]{Lemma}
\newtheorem{corollary}[theorem]{Corollary}
\newdefinition{definition}{Definition}
\newdefinition{exmp}{Example}
\crefname{equation}{}{}
\Crefname{equation}{}{}
\crefname{figure}{Figure}{Fig.}
\crefname{table}{Table}{Table}
\crefname{algorithm}{Algorithm}{Alg.}
\crefname{theorem}{Theorem}{Theorem}
\crefname{lemma}{Lemma}{Lemma}
\crefname{corollary}{Corollary}{Corollary}
\crefname{definition}{Definition}{Definition}
\crefname{exmp}{Example}{Example}
\Crefname{ALC@unique}{Line}{Lines}

\usepackage{enumerate} 
\usepackage{courier} 
\usepackage{setspace} 
\usepackage{soul, soulutf8} 
\usepackage{stfloats} 
\usepackage{url} 

\begin{document}
    \title{{Faster Tensor Train Decomposition for Sparse Data}
    }
    
    
    \author[1]{Lingjie Li}
    \ead{li-lj18@mails.tsinghua.edu.cn}
    \author[1]{Wenjian Yu\corref{cor1}}
    \ead{yu-wj@tsinghua.edu.cn}
    \author[2]{Kim Batselier}
    \ead{k.batselier@tudelft.nl}
    
    \address[1]{Department of Computer Science and Technology, BNRist, Tsinghua University, Beijing 100084, China}
    \address[2]{The Delft Center for Systems and Control, Delft University of Technology, Delft, Netherlands}
    
    \cortext[cor1]{Corresponding author}
    
    \date{}
    
    \begin{abstract}
        In recent years, the application of tensors has become more widespread in fields that involve data analytics and numerical computation. Due to the explosive growth of data, low-rank tensor decompositions have become a powerful tool to harness the notorious \textit{curse of dimensionality}. The main forms of tensor decomposition include CP decomposition, Tucker decomposition, tensor train (TT) decomposition, etc. Each of the existing TT decomposition algorithms, including the TT-SVD and randomized TT-SVD, is successful in the field, but neither can both accurately and efficiently decompose large-scale sparse tensors. Based on previous research, this paper proposes a new quasi-best fast TT decomposition algorithm for large-scale sparse tensors with proven correctness and the upper bound of its complexity is derived. In numerical experiments, we verify that the proposed algorithm can decompose sparse tensors faster than the TT-SVD, and have more speed, precision and versatility than randomized TT-SVD, and it can be used to decomposes arbitrary high-dimensional tensor without losing efficiency when the number of non-zero elements is limited. The new algorithm implements a large-scale sparse matrix TT decomposition that was previously unachievable, enabling tensor decomposition based algorithms to be applied in larger-scale scenarios.
    \end{abstract}
    
    
    
    \maketitle
    
    \section{Introduction}
    
    In the fields of physics, data analytics, scientific computing, digital circuit design, machine learning, etc., data are often organized into a matrix or tensor so that various sophisticated data processing techniques can be applied. One example of such a technique is the low-rank matrix decomposition. It is often implemented through the well-known singular value decomposition (SVD).
    \begin{equation*}
        \mat{A}=\mat{U}\mat{\Sigma}\mat{V}^\top,
    \end{equation*}
    where $\mat{A}\inR{n\times m}, \mat{U}\inR{n\times n}, \mat{\Sigma}\inR{n\times m}, \mat{V}\inR{m\times m}$. $\mat{U}$ and $\mat{V}$ are orthogonal matrices, and $\mat{\Sigma}$ is a diagonal matrix whose diagonal elements (a.k.a. singular values) $\sigma_i, (1\le i\le \min(m, n))$ are non-negative and non-ascending.
    
    
    In recent years, tensors, as a high-dimensional extension of matrices, have also been applied as a powerful and universal tool. In order to overcome the \textit{curse of dimensionality} (the data size of a tensor increases exponentially with the increase of the dimensionality of the tensor), people have extended the notion of a low-rank matrix decomposition to tensors, proposing tensor decompositions such as the CP decomposition \cite{hitchcock1927expression}, the Tucker decomposition \cite{tucker1966some} and the tensor train (TT) decomposition \cite{oseledets2011tensor}. Among them, the TT decomposition transforms the storage complexity of an $n^d$ tensor into $O(dnr^2)$, where $r$ is the maximal TT rank, effectively removing the exponential dependence on $d$. The TT decomposition is advantageous for processing large data sets and has been applied to problems like linear equation solution \cite{oseledets2012solution}, electronic design automation (EDA)~\cite{zhang2014enabling,liu2015model,zhang2017tensor}, system identification~\cite{batselier2017tensor}, large-scale matrix processing~\cite{oseledets2010approximation,kressner2016low,batselier2018computing}, image/video inpainting~\cite{wang2017efficient,ko2018fast}, data mining~\cite{wang2018tensor} and machine learning~\cite{wang2017support,chen2017parallelized,xu2018whole}.
    
    To realize the TT decomposition, the TT-SVD algorithm \cite{oseledets2011tensor} was proposed. It involves a sequence of SVD computations on reshaped matrices. For a large-scale sparse tensor, the TT-SVD consumes excessive computing time and memory usage. Another method employs “cross approximation” to perform low-rank TT-approximations \cite{oseledets2010tt,savostyanov2011fast}, but it still needs too many calculations to find a good representation. Recently, a randomized TT-SVD algorithm \cite{huber2017randomized} was proposed, which incorporates the randomized SVD algorithm \cite{halko2011finding} into the TT-SVD algorithm so as to reduce the runtime for converting a sparse tensor. However, due to the inaccuracy of the randomized SVD, the randomized TT-SVD algorithm usually results in the TT with exaggerated TT ranks or insufficient accuracy. This largely limits its application. 
    
    In this work, we propose a fast and effective TT decomposition algorithm specifically for large sparse data tensors. It includes the steps of constructing an exact TT with nonzero fibers, more efficient parallel-vector rounding and revised TT-rounding. The new algorithm, called \emph{FastTT}, produces the same compact TT representation as the TT-SVD algorithm \cite{oseledets2011tensor}, but exhibits a significant runtime advantage for large sparse data. We  have also extended the algorithm to convert a matrix into the ``matrix in TT-format'', also known as a matrix product operator (MPO). In addition, dynamic approaches are proposed to choose the parameters in the FastTT algorithm. Experiments are carried out on sparse data in problems of image/video inpainting, linear equation solution, and data analysis. The results show that the proposed algorithm is several times to several hundreds times faster than the TT-SVD algorithm without loss of accuracy or an increase of the TT ranks. The speedup ratios are up to 9.6X for the image/video inpainting, 240X for the linear equation and 35X for the sparse data processing, respectively.
    The experimental results also reveal the effectiveness of the proposed dynamic approaches for choosing the parameters in the FastTT algorithm, and the advantages of FastTT over TT-cross and the randomized TT-SVD algorithm \cite{huber2017randomized}. For reproducibility, we have shared the C++ codes of the proposed algorithms and experimental data on \url{https://github.com/lljbash/FastTT}.
    
    
    
    \section{Notations and Preliminaries}
    
    In this article we use boldface capital calligraphic letters (e.g. $\ten{A}$) to denote tensors, boldface capital letters (e.g. $\mat{A}$) to denote matrices, boldface letters (e.g. $\mat{a}$) to denote vectors, and roman (e.g. $a$) or Greek (e.g. $\alpha$) letters to denote scalars.
    
    
    \subsection{Tensor}
    
    Tensors are a high-dimensional generalization of matrices and vectors. A one-dimensional array $\mat{a}\inR{n}$ is called a vector, and a two-dimensional array $\mat{A}\inR{n_1\times n_2}$ is called a matrix. When the dimensionality is extended to $d\ge3$, the $d$-dimensional array $\ten{A}\inR{n_1\times n_2\times\cdots\times n_d}$ is called a $d$-way tensor. The positive integer $d$ is defined as the \textbf{order} of the tensor. $(n_1,n_2,\cdots,n_d)$ are the \textbf{dimensions} of the tensor, where each $n_k$ is the dimension of a particular mode. Vectors and matrices can be considered as 1-way and 2-way tensors, respectively.
    
    \subsection{Basic Tensor Arithmetic}
    
    \begin{definition}{\textbf{Vectorization.}}
        If we rearrange the entries of $\ten{A}\inR{n_1\times\cdots\times n_d}$ into a vector $\mat{b}\inR{\prod_{k=1}^{d}n_k}$, where
        \begin{equation*}
            a_{i_1,i_2,\cdots,i_d}=b_{\sum_{k=1}^{d-1}\left[(i_k-1)\prod_{l=k+1}^{d}n_l\right]+i_d},
        \end{equation*}
        then the vector $\mat{b}$ is called the vectorization of the tensor $\ten{A}$, represented as $\vect(\ten{A})$.
    \end{definition}
    
    \begin{definition}{\textbf{Reshaping.}}
        Like vectorization, if we rearrange the entries of $\ten{A}$ into anther tensor $\ten{B}$ satisfying $\vect(\ten{A})=\vect(\ten{B})$, then the tensor $\ten{B}$ is called the reshaping of $\ten{A}$, represented as $\reshape(\ten{A}, Dims)$ 
        , where $Dims$ denotes the dimensions of $\ten{B}$. In fact, vectorization is a special kind of reshaping.
    \end{definition}
    
    \begin{definition}{\textbf{Unfolding \cite{oseledets2011tensor}.}}
        Unfolding is also a kind of reshaping. If we reshape $\ten{A}\inR{n_1\times n_2\times\cdots\times n_d}$ into a matrix $\mat{B}\inR{m_1\times m_2}$ where $m_1=\prod_{j=1}^{k}n_j,m_2=\prod_{j=k+1}^{d}n_j$, then $\mat{B}$ is called the $k$-unfolding of $\ten{A}$, represented as $\spl_k(\ten{A})$.
    \end{definition}
    
    \begin{definition}{\textbf{Contraction.}}\label{def:contraction}
        Contraction is the tensor generalization of matrix product. For two tensors $\ten{A}\inR{n_1\times n_2\times\cdots\times n_{d_1}}$ and $\ten{B}\inR{m_1\times m_2\times\cdots\times m_{d_2}}$ satisfying $n_{k_1}=m_{k_2}$, their $(k_1,k_2)$-contraction $\ten{C}=\ten{A}\circ_{k_1}^{k_2}\ten{B}$ is defined as
        \begin{equation*}
            c_{i_1\cdots i_{k_1-1}j_1\cdots j_{k_2-1}j_{k_2+1}\cdots j_{d_2}i_{k_1+1}\cdots i_{d_1}}=
            \sum_{l=1}^{n_{k_1}}
            a_{i_1\cdots i_{k_1-1}li_{k_1+1}\cdots i_{d_1}}
            b_{j_1\cdots j_{k_2-1}lj_{k_2+1}\cdots j_{d_2}},
        \end{equation*}
        where $\ten{C}\inR{n_1\times\cdots\times n_{k_1-1}\times m_1\times\cdots\times m_{k_2-1}\times m_{k_2+1}\times\cdots\times m_{d_2}\times n_{k_1+1}\times\cdots\times n_{d_1}}$. If $k_1$ and $k_2$ are not specified, $\ten{A}\circ\ten{B}$ means the $(d_1,1)$-contraction of $\ten{A}$ and $\ten{B}$.
    \end{definition}
    
    \begin{definition}{\textbf{Tensor-matrix product.}}\label{def:t-m-product}
        The $k$-product of a tensor $\ten{A}$ and a matrix $\mat{B}$ can be defined as tensor contraction if the matrix is treated as a 2-way tensor $\ten{B}$.
        \begin{equation*}
            \ten{A}\times_k\mat{B}=\ten{A}\circ_k^1\ten{B}.
        \end{equation*}
    \end{definition}
    
    \begin{definition}{\textbf{Rank-1 tensor.}}\label{def:rank1tt}
        A rank-1 $d$-way tensor can be written as the outer product
        \begin{equation*}
            \ten{A}=\mat{u}^{(1)}\circ\mat{u}^{(2)}\circ\cdots\circ\mat{u}^{(d)},
        \end{equation*}
        of $d$ column vectors $\mat{u}^{(1)}\inR{n_1},\ldots,\mat{u}^{(d)}\inR{n_d}$. The entries of $\ten{A}$ can be computed as $a_{i_1i_2\cdots i_d}=u^{(1)}_{i_1}u^{(2)}_{i_2}\cdots u^{(d)}_{i_d}$.
    \end{definition}
    
    \subsection{Tensor Train Decomposition}
    
    A tensor train decomposition \cite{oseledets2011tensor}, shown in \cref{fig:ttda}, represents a $d$-way tensor $\ten{A}\inR{n_1\times n_2\times\cdots\times n_d}$ with two 2-way tensors and $(d-2)$ 3-way tensors:
    \begin{equation*}
        \ten{A}=\core{G}{1}\circ\core{G}{2}\circ\cdots\circ\core{G}{d},
    \end{equation*}
    where $\core{G}{k}\inR{r_{k-1}\times n_k\times r_k}$ is the $k$-th core tensor. Per definition, $r_0=r_d=1$ such that $\core{G}{1}$ and $\core{G}{d}$ are actually matrices. The dimensions $r_0, r_1, \ldots, r_d$ of the auxiliary indices are called the tensor-train (TT) ranks. When all the TT ranks have the same value, then we can just call it the \emph{TT rank}.
    
    \begin{figure}[!htbp]
        \centering
        \label{fig:ttd}
        \begin{tabular}{cc}
            \subfloat[A 3-way tensor and its TT decomposition] {\includegraphics[width=0.56\textwidth]{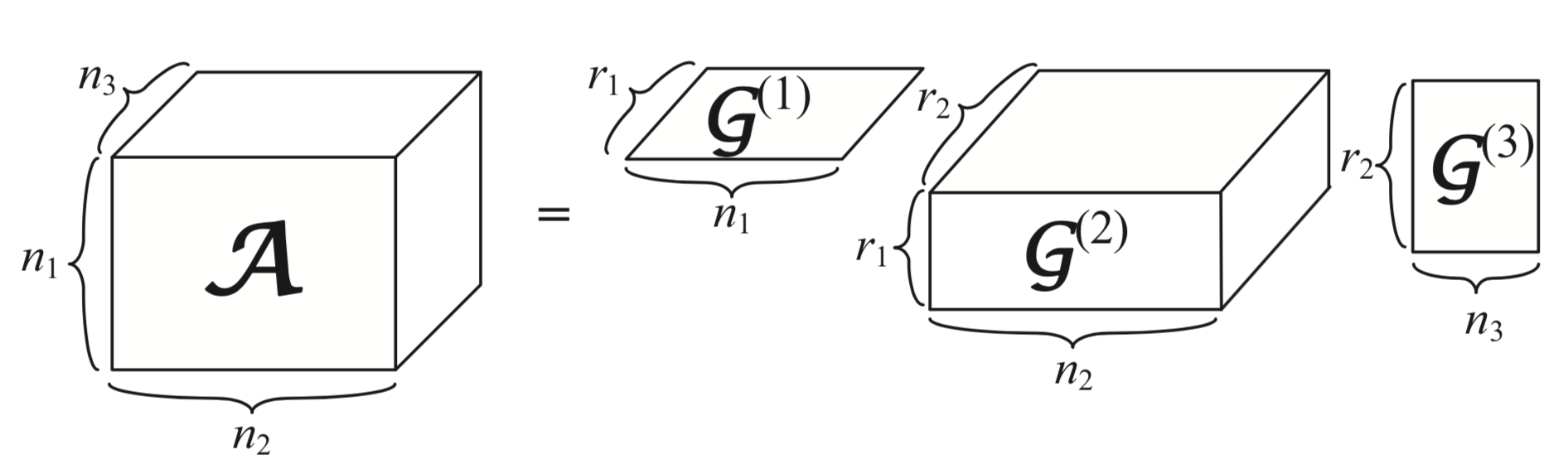}\label{fig:ttda}} &
            \subfloat[TT diagram of a 3-way tensor] {
                \resizebox{0.4\textwidth}{!}{%
                    \begin{tikzpicture}[text height=1.5ex, text depth=0.25ex, auto, node distance=1.5cm, ampersand replacement=\&]
                        \tikzstyle{tensor} = [circle, thick, inner sep=0pt, minimum size=1cm, draw=black]
                        \tikzstyle{null} = [circle, minimum size=0cm]
                        \matrix[column sep=0.75cm] {
                            \node[tensor] (g1) {$\core{G}{1}$}; \&
                            \node[tensor] (g2) {$\core{G}{2}$}; \&
                            \node[tensor] (g3) {$\core{G}{3}$}; \&
                            \\
                        };
                        \node[null] (n0) [left of=g1] {};
                        \node[null] (n1) [below of=g1] {};
                        \node[null] (n2) [below of=g2] {};
                        \node[null] (n3) [below of=g3] {};
                        \node[null] (n4) [right of=g3] {};
                        \path[every node/.style={above}]
                        (n0) edge[dashed] node{$r_0$} (g1)
                        (g1) edge node{$r_1$} (g2)
                        (g2) edge node{$r_2$} (g3)
                        (g3) edge[dashed] node{$r_3$} (n4)
                        ;
                        \path[every node/.style={left}]
                        (g1) edge node{$n_1$} (n1)
                        (g2) edge node{$n_2$} (n2)
                        (g3) edge node{$n_3$} (n3)
                        ;
                    \end{tikzpicture}
                }\label{fig:ttdb}
            }
        \end{tabular}
        \caption{Graphical illustrations of the tensor train (TT) decomposition, where a  3-way tensor $\ten{A}$ is decomposed into two 2-way tensors $\core{G}{1},\core{G}{3}$ and a 3-way tensor $\core{G}{2}$.}
    \end{figure}
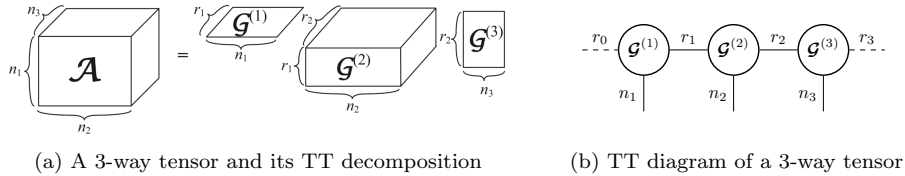
    
    \Cref{fig:ttdb} shows a very convenient graphical representation \cite{batselier2017tensor} of a tensor train. In this diagram, each circle represents a tensor where each ``leg'' attached to it denotes a particular mode of the tensor. The connected line between two circles represents the contraction of two tensors. The dimension is labeled besides each ``leg''. \Cref{fig:ttdb} also illustrates a simple tensor network, which is a collection of tensors that are interconnected through contractions. By fixing the second index of $\core{G}{k}$ to $i_k$, we obtain a matrix $\ten{G}^{(k)}_{i_k}$ (actually a vector if $k=1$ or $k=d$). Then the entries of $\ten{A}$ can be computed as
    \begin{equation*}
        a_{i_1i_2\cdots i_d}=\core{G}{1}_{i_1}\core{G}{2}_{i_2}\cdots\core{G}{d}_{i_d}.
    \end{equation*}
    
    The tensor train decomposition can be computed with the TT-SVD algorithm \cite{oseledets2011tensor}, which consists of doing $d-1$ consecutive reshapings and matrix SVD computations. It is described as \cref{alg:TTSVD}. The expression $\mathrm{rank}_\delta(\mat{C})$ denotes the number of remaining singular values after the $\delta$-truncated SVD. An advantage of TT-SVD is that a quasi-optimal approximation can be obtained with a given error bound and an automatic rank determination.
    
    \begin{algorithm}[!htbp]
        \caption{TT-SVD \cite[p.~2301]{oseledets2011tensor}}
        \label{alg:TTSVD}
        \begin{algorithmic}[1]
            \setcounter{ALC@unique}{0}
            \REQUIRE a tensor $\ten{A}\inR{n_1\times n_2\times\ldots\times n_d}$, desired accuracy tolerance $\varepsilon$.
            \ENSURE Core tensors $\core{G}{1},\ldots,\core{G}{d}$ of the TT-approximation $\ten{B}$ to $\ten{A}$ with TT ranks $r_k$ ($k=0, 1, \cdots, d$) satisfying
            \begin{equation*}
                \|\ten{A}-\ten{B}\|_F\le\varepsilon\|\ten{A}\|_F.
            \end{equation*}
            \STATE Compute truncation parameter $\delta=\frac{\varepsilon}{\sqrt{d-1}}\|\ten{A}\|_F$.
            \STATE $\ten{C}\coloneqq\ten{A},r_0\coloneqq1$.
            \FOR{$k=1$ to $d-1$}
            \STATE $\mat{C}\coloneqq$ reshape$(\ten{C},[r_{k-1}n_k,\prod_{i=k+1}^{d}n_i])$.
            \STATE Compute $\delta$-truncated SVD: $\mat{C}=\mat{U\Sigma V}^T+\mat{E}$, $\|\mat{E}\|_F\le\delta$, ~ $r_k\coloneqq\rank_\delta(\mat{C})$.
            \STATE $\core{G}{k}\coloneqq\reshape(\mat{U},[r_{k-1},n_k,r_k])$.
            \STATE $\ten{C}\coloneqq \mat{\Sigma V}^T$.
            \ENDFOR
            \STATE $\core{G}{d}\coloneqq\ten{C}$
            \STATE Return tensor $\ten{B}$ in TT-format with cores $\core{G}{1},\ldots,\core{G}{d}$.
        \end{algorithmic}
    \end{algorithm}
    
    We define the FLOP count of the TT-SVD algorithm as $f_{\mathrm{TTSVD}}$. Then
    
    \begin{equation}\label{equ:ttsvd}
        f_{\mathrm{TTSVD}}\approx\sum_{i=1}^{d-1}\left[f_{\mathrm{SVD}}\left(r_{i-1}n_i,\prod_{j=i+1}^{d}n_j\right)\right],
    \end{equation}
    where $f_{\mathrm{SVD}}(m,n)=C_{\svd}\, m\,n\min(m,n)$ is the FLOP count of performing the economic SVD for an $m\times n$ dense matrix. 
    
    A big problem with \cref{alg:TTSVD} is the large computation cost of $\delta$-truncated SVD on large-scale unfolded matrices when the dimensions grow. A possible solution is to replace SVD with more economic decomposition like pseudo-skeleton decomposition~\cite{goreinov1997theory}. TT-cross \cite{oseledets2010tt,savostyanov2011fast} is a multidimensional generalization of the skeleton decomposition to the tensor case. The algorithm uses a sweep strategy and can produce TT-approximates with given accuracy or maximal ranks. The time complexity of TT-cross depends on $d$ linearly.
    
    As for decomposing large-scale sparse tensors, a simple idea is to employ the truncated SVD algorithm based on Krylov subspace iterative method, e.g. the built-in function \texttt{svds} in Matlab. However, \texttt{svds} requests a truncation rank as input, and thus cannot be directly applied here. Moreover, the sparsity can only be taken advantage of at the first iteration step of \cref{alg:TTSVD}. After the first truncated SVD, the  tensor becomes dense and thus the computation cannot be reduced.

    \subsection{Rounding}
    
    Sometimes one is given tensor data already in the TT-format but with suboptimal TT-ranks. In order to save storage and speed up the following computation, one can reduce the TT-ranks while maintaining accuracy, through a procedure called \textit{rounding}. It is realized with the TT-rounding  algorithm \cite{oseledets2011tensor} (\cref{alg:ttro}). The algorithm is based on the same principle as TT-SVD and also produces quasi-optimal TT-ranks with a given error bound. TT-rounding can be of great use in cases where a large tensor is represented in TT-format.
    \begin{algorithm}[!htbp]
        \caption{TT-rounding \cite[p.~2305]{oseledets2011tensor}}
        \label{alg:ttro}
        \begin{algorithmic}[1]
            \setcounter{ALC@unique}{0}
            \REQUIRE Cores $\core{G}{1},\ldots,\core{G}{d}$ of the TT-format tensor $\ten{A}$ with TT-ranks $r_1,\ldots,r_{d-1}$, desired accuracy tolerance $\varepsilon$.
            \ENSURE Cores $\core{G}{1},\ldots,\core{G}{d}$ of the TT-approximation $\ten{B}$ to $\ten{A}$ in the TT-format with TT-ranks $r_1,\ldots,r_{d-1}$. The computed approximation satisfies
            \begin{equation*}
                ||\ten{A}-\ten{B}||_F\le\varepsilon||\ten{A}||_F.
            \end{equation*}
            \STATE Compute truncation parameter $\delta=\frac{\varepsilon}{\sqrt{d-1}}\|\ten{A}\|_F$.
            \FOR{$k=d,\dots,2$}
            \STATE $\mat{G}\coloneqq\spl_1^T(\core{G}{k}).$
            \STATE Compute economic QR decomposition $\mat{G}=\mat{QR},\quad r_{k-1}\coloneqq\rank(\mat{G})$.
            \STATE $\core{G}{k}\coloneqq\reshape(\mat{Q^T},[r_{k-1},n_k,r_k]),\quad\core{G}{k-1}\coloneqq\core{G}{k-1}\times_3 \mat{R}^T.$
            \ENDFOR
            \FOR{$k=1,\dots,d-1$}
            \STATE $\mat{G}\coloneqq\spl_2(\core{G}{k})$.
            \STATE Compute $\delta$-truncated SVD: $\mat{G}=\mat{U\Sigma V}^T+\mat{E}$, $\|\mat{E}\|_F\le\delta$, ~ $r_k\coloneqq\rank_\delta(\mat{C})$.
            \STATE $\core{G}{k}\coloneqq\reshape(\mat{U},[r_{k-1},n_k,r_k]),\quad\core{G}{k+1}\coloneqq\core{G}{k+1}\times_1(\mat{\Sigma V}^T)$.
            \ENDFOR
            \STATE $\core{G}{d}\coloneqq\ten{C}$
            \STATE Return tensor $\ten{B}$ in TT-format with cores $\core{G}{1},\ldots,\core{G}{d}$.
        \end{algorithmic}
    \end{algorithm}
    
    Parallel-vector rounding~\cite{hubig2017generic} is another rounding method which replaces the truncated SVD with Deparallelisation (\cref{alg:depar}). It removes the paralleled columns of an $a\times b$ matrix in $O(ab\alpha)$ time, where $\alpha$ is the number of non-parallel columns. Parallel-vector rounding is lossless, runs much faster than TT-rounding and can preserve the sparsity of TT. However, it usually cannot reduce the TT-ranks much; its effectiveness highly depends on the parallelism in TT-cores. Therefore, the parallel-vector rounding is suitable for a constructed sparse TT, rather than the construction of a TT.
    
    \begin{algorithm}[!htbp]
        \caption{Deparallelisation~\cite[Appendix B]{hubig2017generic}}
        \label{alg:depar}
        \begin{algorithmic}[1]
            \setcounter{ALC@unique}{0}
            \REQUIRE Matrix $\mat{M}\inR{a\times b}$.
            \ENSURE Matrix $\mat{N}\inR{a\times\beta},\mat{T}\inR{\beta\times b}$ s.t. $\mat{M}=\tilde{\mat{M}}\times\mat{T}$ and $\mat{N}$ has at most as many columns as $\mat{M}$ and no two columns which are parallel to each other.
            \STATE Let $K$ be the set of kept columns, empty initially.
            \STATE Let $\mat{T}$ be the dynamically-resized transfer matrix.
            \FOR{every column index $j\in[1,b]$}
            \FOR{every kept index $i\in[1,|K|]$}
            \IF{the $j$-th column $\mat{M}_{:j}$ is parallel to column $K_i$}
            \STATE Set $T_{i,j}$ to the prefactor between the two columns.
            \ELSE
            \STATE add $\mat{M}_{:j}$ to $K$, set $T_{|K|,j}=1$.
            \ENDIF
            \ENDFOR
            \ENDFOR
            \STATE Construct $\mat{N}$ by horizontally concatenating the columns stored in $K$.
            \STATE Return $\mat{N}$ and $\mat{T}$.
        \end{algorithmic}
    \end{algorithm}
    
    \section{Faster tensor train decomposition of sparse tensor}
    The TT-SVD algorithm does not take advantage of the possible sparsity of data since the $\delta$-truncated SVD is used. In this section, we propose a new algorithm for computing the TT decomposition of a sparse tensor whereby the sparsity is explicitly exploited. The key idea is to rearrange the data in such a way that the desired TT decomposition can be written down explicitly, followed by a parallell-vector rounding step.
    
    \subsection{Constructing TT with nonzero $p$-fibers}
    
    \begin{figure}[!htbp]
        \centering
        \includegraphics[width=\textwidth]{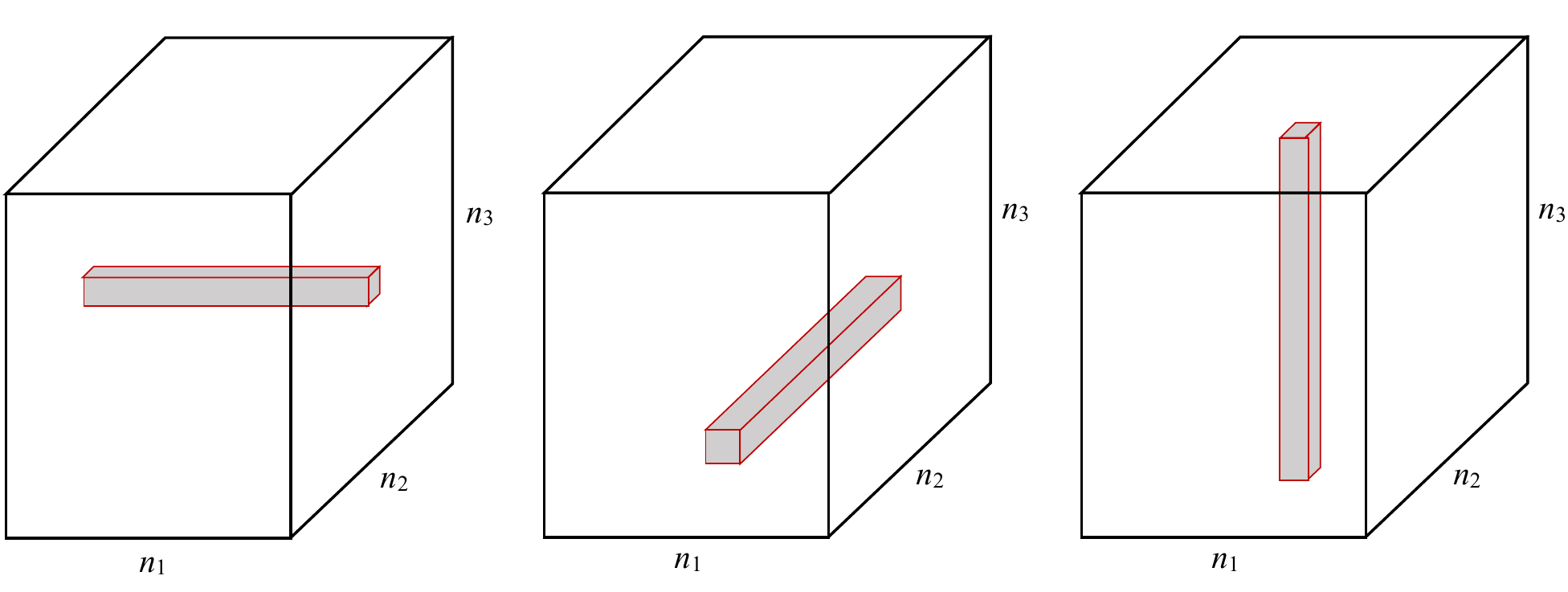}
        \vskip -10pt
        \caption{A $1$-fiber, a $2$-fiber and a $3$-fiber of tensor $\ten{A}\inR{n_1\times n_2\times n_3}$.}
        \label{fig:pfiber}
    \end{figure}
    
    \overfullrule=0pt For a tensor \mbox{$\ten{A} \inR{n_1\times\cdots\times n_d}$} and a given integer $p$ that satisfies $1\le p\le d$, we define a $p$-fiber of a tensor as a fiber of the tensor in the direction $\mat{e}_p$. \Cref{fig:pfiber} shows a $1$-fiber, a $2$-fiber and a $3$-fiber of a $3$-d tensor. We can specify a $p$-fiber of $\ten{A}$ by fixing the indices except the $p$-th dimension $\mat{i}_p=(i_1, \ldots, i_{p-1}, i_{p+1}, \ldots, i_d)$, 
    \begin{equation}
        \mat{v}_p(\mat{i}_p) \coloneqq \ten{A}(i_1, \ldots, i_{p-1}, :, i_{p+1}, \ldots, i_d) .
    \end{equation}
    Then a tensor $\ten{A}_1$ with only one non-zero $p$-fiber $\mat{v}_p(\mat{i}_p)$ can be easily formed as the outer product of the $p$-fiber and $d-1$ standard basis vector and is thus a rank-1 tensor,
    \begin{equation}
        \ten{A}_1=\mat{e}_{i_1}\circ\cdots\circ\mat{e}_{i_{p-1}}\circ \mat{v}_p(\mat{i}_p)\circ \mat{e}_{i_{p+1}}\circ\cdots\circ \mat{e}_{i_d}.
    \end{equation}
    
    Suppose we have $R$ nonzero $p$-fiber of sparse tensor $\ten{A}$ with indices $(i_1,\ldots,i_{p-1},i_{p+1},\ldots,i_d)$ forming a set $S_p$ with $|S_p|=R$. Then, $\ten{A}$ can be represented as the sum of $R$ rank-1 tensors,
    \begin{equation}\label{equ_split}
        \ten{A}=\!\!\sum_{(i_1,\ldots,i_{p-1},i_{p+1},\ldots,i_d)\in S_p}\!\!\mat{e}_{i_1}\circ\cdots\circ \mat{e}_{i_{p-1}}\circ \mat{v}_{i_1,\ldots,i_{p-1},i_{p+1},\ldots,i_d}\circ \mat{e}_{i_{p+1}}\circ\cdots\circ \mat{e}_{i_d},
    \end{equation}
    where $\mat{e}_{i_k}\inR{n_k}$ is the standard basis vector. Next, we are going to construct a tensor train based on this representation.
    
    \begin{lemma}\label{lem:rank1}
        \mbox{Any rank-1 tensor is equivalent to a tensor train whose TT rank is 1.}
        \begin{equation}\label{equ_rank1}
            \mat{v}_1\circ \mat{v}_2\circ\ldots\circ \mat{v}_d=\core{V}{1}\circ\core{V}{2}\circ\ldots\circ\core{V}{d},
        \end{equation}
        \overfullrule=0pt   where $\mat{v}_k \inR{n_k}$, $(k=1, \cdots, d)$, $\core{V}{1}=\reshape(\mat{v}_1,[n_1,1])$, $\core{V}{d}=\reshape(\mat{v}_d,[1,n_d])$, and $\core{V}{k}=\reshape(\mat{v}_k,[1,n_k,1])$ for $1<k<d$.
    \end{lemma}
    
    \begin{lemma} \cite[p.~2308]{oseledets2011tensor}\label{lem:addtt}
        Suppose we have two tensors $\ten{A}\inR{n_1\times n_2\times\ldots\times n_d}$ and $\ten{B}\inR{n_1\times n_2\times\ldots\times n_d}$ in the TT format,
        \begin{align*}
            a_{i_1i_2\cdots i_d}&=\core{A}{1}_{i_1}\core{A}{2}_{i_2}\cdots\core{A}{d}_{i_d},\\
            b_{i_1i_2\cdots i_d}&=\core{B}{1}_{i_1}\core{B}{2}_{i_2}\cdots\core{B}{d}_{i_d}.
        \end{align*}
        The TT cores of the sum $\ten{C}=\ten{A}+\ten{B}$ in the TT format then satisfy
        \begin{align}\label{equ_ttsum}
            \begin{split}
                \core{C}{k}_{i_k} &= \begin{bmatrix} \core{A}{k}_{i_k} & \mat{O} \\ \mat{O} & \core{B}{k}_{i_k} \end{bmatrix},\ k = 2,\ldots,d-1,\\
                \core{C}{1}_{i_1} &= \begin{bmatrix} \core{A}{1}_{i_1} & \core{B}{1}_{i_1} \end{bmatrix},  ~ 
                \core{C}{d}_{i_d} = \begin{bmatrix} \core{A}{d}_{i_d}\\[2ex] \core{B}{d}_{i_d} \end{bmatrix}, 
            \end{split}
        \end{align}
        where $\mat{O}$ denotes a zero matrix of appropriate dimensions.
    \end{lemma}
    The proof of \cref{lem:rank1,lem:addtt} can be easily derived from \cref{def:contraction,def:t-m-product,def:rank1tt} and the definition of the tensor train decomposition.

    Based on \cref{equ_split,lem:rank1,lem:addtt}, we have the following theorem.
    
    \begin{theorem}\label{thm:trans}
        \overfullrule=0pt    A sparse tensor $\ten{A}\inR{n_1\times n_2\times\ldots\times n_d}$ can be transformed into an equivalent tensor train with TT rank $R$, where $R$ is the number of nonzero $p$-fiber in $\ten{A}$ $(1\le p \le d)$. If $p \neq 1$ or $d$,
        \begin{equation}
            \label{eqn:rankR}
            \ten{A}=\core{P}{1}\circ\ldots\circ\core{P}{p-1}\circ\ten{V}\circ\core{P}{p+1}\circ\ldots\circ\core{P}{d},
        \end{equation}
        where
        $\core{P}{k}\in\{0,1\}^{R\times n_j\times R}, (2\le k \le d, k\ne p)$, 
        $\core{P}{1}\in\{0,1\}^{1\times n_1\times R}$, $\core{P}{d}\in\{0,1\}^{R\times n_d\times1}$, and $\ten{V}\inR{R\times n_p\times R}$.
        Similar expressions hold for the situations with $p=1$ or $d$.
    \end{theorem}
    
    The TT cores $\core{P}{k}$ and $\ten{V}$ in \cref{thm:trans} are sparse tensors, whose nonzero distributions are illustrated in \cref{fig:pv}. Each horizontal bar depicted in \cref{fig:pv} is a standard basis vector $\mat{e}_{i_k}$ for $\core{P}{k}$ or a $p$-fiber $\mat{v}$ for $\ten{V}$. The derived matrices ($\core{P}{k}_{i_k}$ and $\ten{V}_{i_p}$) from these TT cores are all diagonal matrices. Furthermore, each of the $\core{P}{k}$ cores is very sparse, as the nonzero elements consist of only $R$ 1's.
    \begin{figure}[!htbp]
        \centering
        \subfloat[$\core{P}{k},(2\le k \le d, k\ne p)$] {\includegraphics[height=1.85in]{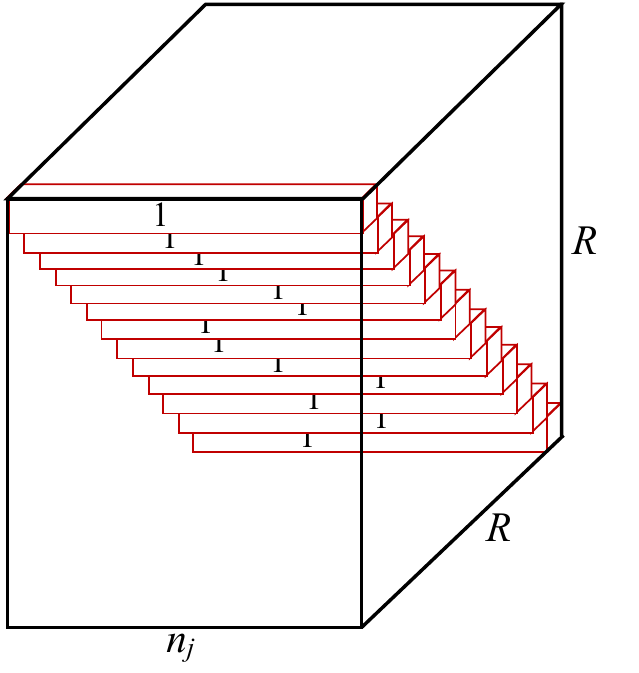}}
        \subfloat[$\ten{V}$] {\includegraphics[height=1.85in]{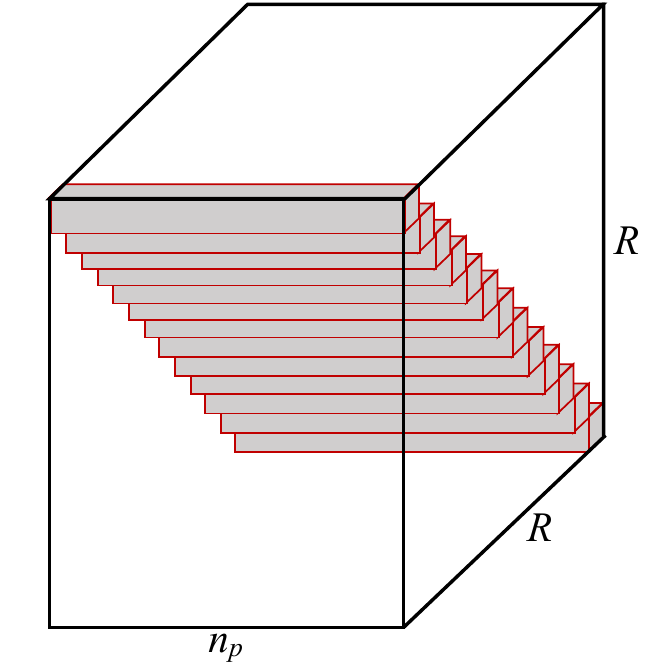}}
        \caption{The nonzero distributions of $\core{P}{k},(2\le k \le d, k\ne p)$ and $\ten{V}$.}
        \label{fig:pv}
    \end{figure}
    
    \subsection{More efficient parallel-vector rounding}
    
    From \cref{fig:pv}, it is obvious that $\core{P}{k}$ are very sparse and the elements are arranged regularly. Hence parallel-vector rounding should be very effective on the TT obtained in \cref{thm:trans}. In order to maximize the effect of Deparallelisation, we modify the original algorithm so that the decomposition on the less regular core $\ten{V}$ is avoided. By combining this modified algorithm with \cref{thm:trans}, we obtain a lossless sparse tensor to TT conversion algorithm, described as \cref{alg:dedup_rounding}, where \textit{Depar} refers to the Deparallelisation algorithm.
    
    \begin{algorithm}[!htbp]
        \caption{Sparse TT conversion with modified parallel-vector rounding}
        \label{alg:dedup_rounding}
        \begin{algorithmic}[1]
            \setcounter{ALC@unique}{0}
            \REQUIRE A sparse tensor $\ten{A}\inR{n_1\times n_2\times\ldots\times n_d}$, an integer $p$ ($1\le p \le d$).
            \ENSURE Core tensors $\core{G}{1},\ldots,\core{G}{d}$ of TT-format tensor $\ten{B}$ which is equivalent to $\ten{A}$ with TT ranks $\tilde{r}_k$ ($k=0, 1, \cdots, d$).
            \STATE Initialize empty cores $\core{G}{1},\ldots,\core{G}{d}$ for TT-format tensor $\ten{B}$.
            \FOR{every $\mat{v}\in$ \{all $R$ nonzero $p$-subvectors of $\ten{A}$\}}
            \STATE Determine $(d-1)$ $\mat{e}_i$ vectors in \cref{equ_split}.
            \STATE Construct rank-1 TT $\ten{T}$ with $\mat{v}$ and $\mat{e}$ vectors as \cref{lem:rank1}.
            \STATE $\ten{B}\coloneqq\ten{B}+\ten{T}$, which means $\core{G}{1},\ldots,\core{G}{d}$ are update with \cref{equ_ttsum}.
            \ENDFOR
            \STATE $\tilde{r}_0\coloneqq1$.
            \FOR{$k=1,\ldots,p-1$}
            \STATE $[\mat{N},\mat{T}]\coloneqq\operatorname{Depar}(\spl_2(\core{G}{k}))$, where $\mat{N}\inR{\tilde{r}_{k-1}n_k\times \tilde{r}_k},\mat{T}\inR{\tilde{r}_k\times R}$.
            \STATE $\core{G}{k}\coloneqq\reshape(\mat{N},[\tilde{r}_{k-1},n_k,\tilde{r}_k])$.
            \STATE $\core{G}{k+1}\coloneqq\core{G}{k+1}\times_1\mat{T}^T$.
            \ENDFOR
            \STATE $\tilde{r}_d\coloneqq1$.
            \FOR{$k=d,\ldots,p+1$}
            \STATE $[\mat{N},\mat{T}]\coloneqq\operatorname{Depar}(\spl_1^T(\core{G}{k}))$, where $\mat{N}\inR{\tilde{r}_kn_k\times \tilde{r}_{k-1}},\mat{T}\inR{\tilde{r}_{k-1}\times R}$.
            \STATE $\core{G}{k}\coloneqq\reshape(\mat{N}^T,[\tilde{r}_{k-1},n_k,\tilde{r}_k])$.
            \STATE $\core{G}{k-1}\coloneqq\core{G}{k-1}\times_3\mat{T}^T$.
            \ENDFOR
            \STATE Return tensor $\ten{B}$ in TT-format with cores $\core{G}{1},\ldots,\core{G}{d}$.
        \end{algorithmic}
    \end{algorithm}
    
    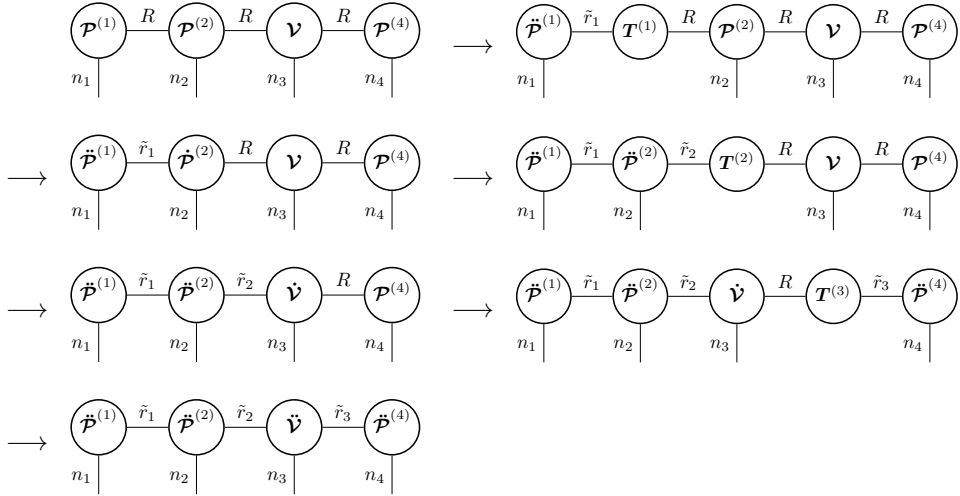
\begin{figure}[!htbp]
        \centering
        \tikzstyle{tensor} = [circle, thick, inner sep=0pt, minimum size=1cm, draw=black]
        \tikzstyle{null} = [circle, minimum size=0cm]
        \begingroup
        \renewcommand*{\arraystretch}{0}
        \begin{tabularx}{\textwidth}{Xp{0.39\textwidth}Xp{0.48\textwidth}}
            &
            \resizebox{0.45\textwidth}{!}{%
                \begin{tikzpicture}[text height=1.5ex, text depth=0.25ex, auto, node distance=1.5cm, ampersand replacement=\&]
                    \matrix[column sep=0.75cm] {
                        \node[tensor] (g1) {$\core{P}1$}; \&
                        \node[tensor] (g2) {$\core{P}2$}; \&
                        \node[tensor] (g3) {$\ten{V}$}; \&
                        \node[tensor] (g4) {$\core{P}4$}; \&
                        \\
                    };
                    \node[null] (n1) [below of=g1] {};
                    \node[null] (n2) [below of=g2] {};
                    \node[null] (n3) [below of=g3] {};
                    \node[null] (n4) [below of=g4] {};
                    \path[every node/.style={above}]
                    (g1) edge node{$R$} (g2)
                    (g2) edge node{$R$} (g3)
                    (g3) edge node{$R$} (g4)
                    ;
                    \path[every node/.style={left}]
                    (g1) edge node{$n_1$} (n1)
                    (g2) edge node{$n_2$} (n2)
                    (g3) edge node{$n_3$} (n3)
                    (g4) edge node{$n_4$} (n4)
                    ;
                \end{tikzpicture}
            } &
            \raisebox{1cm}{$\longrightarrow$} &
            \raisebox{0cm}{\resizebox{0.55\textwidth}{!}{%
                    \begin{tikzpicture}[text height=1.5ex, text depth=0.25ex, auto, node distance=1.5cm, ampersand replacement=\&]
                        \matrix[column sep=0.75cm] {
                            \node[tensor] (g1) {$\core{\ddot{P}}1$}; \&
                            \node[tensor] (q) {$\mat{T}^{(1)}$}; \&
                            \node[tensor] (g2) {$\core{P}2$}; \&
                            \node[tensor] (g3) {$\ten{V}$}; \&
                            \node[tensor] (g4) {$\core{P}4$}; \&
                            \\
                        };
                        \node[null] (n1) [below of=g1] {};
                        \node[null] (n2) [below of=g2] {};
                        \node[null] (n3) [below of=g3] {};
                        \node[null] (n4) [below of=g4] {};
                        \path[every node/.style={above}]
                        (g1) edge node{$\tilde{r}_1$} (q)
                        (q) edge node{$R$} (g2)
                        (g2) edge node{$R$} (g3)
                        (g3) edge node{$R$} (g4)
                        ;
                        \path[every node/.style={left}]
                        (g1) edge node{$n_1$} (n1)
                        (g2) edge node{$n_2$} (n2)
                        (g3) edge node{$n_3$} (n3)
                        (g4) edge node{$n_4$} (n4)
                        ;
                    \end{tikzpicture}
            }} \\
            \raisebox{1cm}{$\longrightarrow$} &
            \resizebox{0.45\textwidth}{!}{%
                \begin{tikzpicture}[text height=1.5ex, text depth=0.25ex, auto, node distance=1.5cm, ampersand replacement=\&]
                    \matrix[column sep=0.75cm] {
                        \node[tensor] (g1) {$\core{\ddot{P}}1$}; \&
                        \node[tensor] (g2) {$\core{\dot{P}}2$}; \&
                        \node[tensor] (g3) {$\ten{V}$}; \&
                        \node[tensor] (g4) {$\core{P}4$}; \&
                        \\
                    };
                    \node[null] (n1) [below of=g1] {};
                    \node[null] (n2) [below of=g2] {};
                    \node[null] (n3) [below of=g3] {};
                    \node[null] (n4) [below of=g4] {};
                    \path[every node/.style={above}]
                    (g1) edge node{$\tilde{r}_1$} (g2)
                    (g2) edge node{$R$} (g3)
                    (g3) edge node{$R$} (g4)
                    ;
                    \path[every node/.style={left}]
                    (g1) edge node{$n_1$} (n1)
                    (g2) edge node{$n_2$} (n2)
                    (g3) edge node{$n_3$} (n3)
                    (g4) edge node{$n_4$} (n4)
                    ;
                \end{tikzpicture}
            } &
            \raisebox{1cm}{$\longrightarrow$} &
            \raisebox{0cm}{\resizebox{0.55\textwidth}{!}{%
                    \begin{tikzpicture}[text height=1.5ex, text depth=0.25ex, auto, node distance=1.5cm, ampersand replacement=\&]
                        \matrix[column sep=0.75cm] {
                            \node[tensor] (g1) {$\core{\ddot{P}}1$}; \&
                            \node[tensor] (g2) {$\core{\ddot{P}}2$}; \&
                            \node[tensor] (q) {$\mat{T}^{(2)}$}; \&
                            \node[tensor] (g3) {$\ten{V}$}; \&
                            \node[tensor] (g4) {$\core{P}4$}; \&
                            \\
                        };
                        \node[null] (n1) [below of=g1] {};
                        \node[null] (n2) [below of=g2] {};
                        \node[null] (n3) [below of=g3] {};
                        \node[null] (n4) [below of=g4] {};
                        \path[every node/.style={above}]
                        (g1) edge node{$\tilde{r}_1$} (g2)
                        (g2) edge node{$\tilde{r}_2$} (q)
                        (q) edge node{$R$} (g3)
                        (g3) edge node{$R$} (g4)
                        ;
                        \path[every node/.style={left}]
                        (g1) edge node{$n_1$} (n1)
                        (g2) edge node{$n_2$} (n2)
                        (g3) edge node{$n_3$} (n3)
                        (g4) edge node{$n_4$} (n4)
                        ;
                    \end{tikzpicture}
            }} \\
            \raisebox{1cm}{$\longrightarrow$} &
            \resizebox{0.45\textwidth}{!}{%
                \begin{tikzpicture}[text height=1.5ex, text depth=0.25ex, auto, node distance=1.5cm, ampersand replacement=\&]
                    \matrix[column sep=0.75cm] {
                        \node[tensor] (g1) {$\core{\ddot{P}}1$}; \&
                        \node[tensor] (g2) {$\core{\ddot{P}}2$}; \&
                        \node[tensor] (g3) {$\ten{\dot{V}}$}; \&
                        \node[tensor] (g4) {$\core{P}4$}; \&
                        \\
                    };
                    \node[null] (n1) [below of=g1] {};
                    \node[null] (n2) [below of=g2] {};
                    \node[null] (n3) [below of=g3] {};
                    \node[null] (n4) [below of=g4] {};
                    \path[every node/.style={above}]
                    (g1) edge node{$\tilde{r}_1$} (g2)
                    (g2) edge node{$\tilde{r}_2$} (g3)
                    (g3) edge node{$R$} (g4)
                    ;
                    \path[every node/.style={left}]
                    (g1) edge node{$n_1$} (n1)
                    (g2) edge node{$n_2$} (n2)
                    (g3) edge node{$n_3$} (n3)
                    (g4) edge node{$n_4$} (n4)
                    ;
                \end{tikzpicture}
            } &
            \raisebox{1cm}{$\longrightarrow$} &
            \raisebox{0cm}{\resizebox{0.55\textwidth}{!}{%
                    \begin{tikzpicture}[text height=1.5ex, text depth=0.25ex, auto, node distance=1.5cm, ampersand replacement=\&]
                        \matrix[column sep=0.75cm] {
                            \node[tensor] (g1) {$\core{\ddot{P}}1$}; \&
                            \node[tensor] (g2) {$\core{\ddot{P}}2$}; \&
                            \node[tensor] (g3) {$\ten{\dot{V}}$}; \&
                            \node[tensor] (q) {$\mat{T}^{(3)}$}; \&
                            \node[tensor] (g4) {$\core{\ddot{P}}4$}; \&
                            \\
                        };
                        \node[null] (n1) [below of=g1] {};
                        \node[null] (n2) [below of=g2] {};
                        \node[null] (n3) [below of=g3] {};
                        \node[null] (n4) [below of=g4] {};
                        \path[every node/.style={above}]
                        (g1) edge node{$\tilde{r}_1$} (g2)
                        (g2) edge node{$\tilde{r}_2$} (g3)
                        (g3) edge node{$R$} (q)
                        (q) edge node{$\tilde{r}_3$} (g4)
                        ;
                        \path[every node/.style={left}]
                        (g1) edge node{$n_1$} (n1)
                        (g2) edge node{$n_2$} (n2)
                        (g3) edge node{$n_3$} (n3)
                        (g4) edge node{$n_4$} (n4)
                        ;
                    \end{tikzpicture}
            }} \\
            \raisebox{1cm}{$\longrightarrow$} &
            \resizebox{0.45\textwidth}{!}{%
                \begin{tikzpicture}[text height=1.5ex, text depth=0.25ex, auto, node distance=1.5cm, ampersand replacement=\&]
                    \matrix[column sep=0.75cm] {
                        \node[tensor] (g1) {$\core{\ddot{P}}1$}; \&
                        \node[tensor] (g2) {$\core{\ddot{P}}2$}; \&
                        \node[tensor] (g3) {$\ten{\ddot{V}}$}; \&
                        \node[tensor] (g4) {$\core{\ddot{P}}4$}; \&
                        \\
                    };
                    \node[null] (n1) [below of=g1] {};
                    \node[null] (n2) [below of=g2] {};
                    \node[null] (n3) [below of=g3] {};
                    \node[null] (n4) [below of=g4] {};
                    \path[every node/.style={above}]
                    (g1) edge node{$\tilde{r}_1$} (g2)
                    (g2) edge node{$\tilde{r}_2$} (g3)
                    (g3) edge node{$\tilde{r}_3$} (g4)
                    ;
                    \path[every node/.style={left}]
                    (g1) edge node{$n_1$} (n1)
                    (g2) edge node{$n_2$} (n2)
                    (g3) edge node{$n_3$} (n3)
                    (g4) edge node{$n_4$} (n4)
                    ;
                \end{tikzpicture}
            }
        \end{tabularx}
        \endgroup
        \caption{The graphical representations of the decomposition forms during the \cref{alg:dedup_rounding} execution for a $4$-way TT $(p=3)$}\label{fig:dedup}
    \end{figure}
    
    The correctness of \cref{alg:dedup_rounding} is due to \cref{thm:trans} and the associative property of matrix multiplications. The graphical representations of the decomposition forms during the algorithm execution are shown in \cref{fig:dedup} for a 4-way TT.
    
    The original Deparallelisation algorithm (\cref{alg:depar}) dose not consider the special pattern of the core tensors. An important observation from \cref{fig:pv} is that each $\core{P}{k}$ obtained by \cref{thm:trans} is consisted of $R$ ``diagonally'' $2$-fiber with only one ``1''. This means all unfoldings $\spl_2(\core{P}{k})$ for $k<p$ and $\spl_1^T(\core{P}{k})$ for $k>p$ are so-called \textit{quasi-permutation matrices}.
    
    \begin{definition}{\textbf{Quasi-permutation matrix.}}\label{def:qp}
        If each column of a matrix has only one nonzero element with a value of 1, then the matrix is called a quasi-permutation matrix. A quasi-permutation matrix $\mat{A}\inR{n\times m}$ can be represented as
        \begin{align}
            \mat{A}=\begin{bmatrix}
                \mat{e}_{i_1}&\mat{e}_{i_2}&\cdots&\mat{e}_{i_m}
            \end{bmatrix},
        \end{align}
        where $e_k$ denotes a standard basis vector in the $n$-dimensional Euclidean space $\mathbb{R}^n$ with a 1 in the $k$th coordinate and 0's elsewhere.
    \end{definition}
    \begin{exmp}
        Obviously, a permutation or identity matrix belongs to the class of quasi-permutation matrices.
    \end{exmp}
    \begin{corollary}\label{cor:qp}
        For \cref{alg:depar}, if the input matrix $\mat{M}$ is a quasi-permutation matrix, matrix $\mat{N}$ and the transfer matrix $\mat{T}$ will also be quasi-permutation matrices.
    \end{corollary}
    
    It turns out that this property of $\core{P}{k}$ is maintained throughout the whole process of parallel-vector rounding, which will enable a more efficient implementation of Deparallelisation.
    
    

    \begin{theorem}\label{thm:dedup}
        In \cref{alg:dedup_rounding}, each input matrix of the function \textbf{Depar} is a quasi-permutation matrix.
    \end{theorem}
    
    \begin{proof}
        $\core{G}{1}$ is definitely a quasi-permutation matrix according to \cref{equ_rank1,equ_ttsum}. For $2\le k<p$, $\core{G}{k}$ changes twice during \cref{alg:dedup_rounding} --- once in iteration $k-1$ and once in iteration $k$. Like in \cref{fig:dedup}, we use $\core{P}{k}$, $\core{\dot{P}}{k}$ and $\core{\ddot{P}}{k}$ to denote the three stages of $\core{G}{k}$. The input matrix of \textit{Depar} in iteration $k$ is $\spl_2(\core{\dot{P}}{k})$. $\core{\dot{P}}{k}$ is computed as $\core{P}{k}\times_1\mat{T}^T$ in iteration $k-1$, which is equivalent to
        \begin{align}
            \spl_1(\core{\dot{P}}{k})=\mat{T}\times\spl_1(\core{P}{k}).
        \end{align}
        
        We can deduce from \cref{fig:pv} that $\spl_1(\core{P}{k})$ has the following structure
        \begin{equation*}
            \left[\begin{array}{ccccccccccccc}
                x_{11} & 0 & \cdots & 0 & x_{12} & 0 & \cdots & 0 & \cdots & x_{1n_k} & 0 & \cdots & 0 \\
                0 & x_{21} & \cdots & 0 & 0 & x_{22} & \cdots & 0 & \cdots & 0 & x_{2n_k} & \cdots & 0 \\
                \vdots & \vdots & \ddots & \vdots & \vdots & \vdots & \ddots & \vdots & \cdots & \vdots & \vdots & \ddots & \vdots \\
                0 & 0 & \cdots & x_{R1} & 0 & 0 & \cdots & x_{R2} & \cdots & 0 & 0 & \cdots & x_{Rn_k} \\
            \end{array}\right],
        \end{equation*}
        where $\forall j=1,2,\ldots,R$, vector $[x_{j1}~x_{j2}~\cdots~x_{jn_k}]^T$ is a particular standard basis vector. Let $\mat{T}=[\mat{e}_{i_1}~\mat{e}_{i_2}~\cdots~\mat{e}_{i_R}]$ according to \cref{cor:qp}. Then $\spl_1(\core{\dot{P}}{k})$ will have the structure
        \begin{equation*}
            \begin{bmatrix}
                x_{11}\mat{e}_{i_1} & x_{21}\mat{e}_{i_2} & \cdots & x_{R1}\mat{e}_{i_R} & \cdots & x_{1n_d}\mat{e}_{i_1} & x_{2n_d}\mat{e}_{i_2} & \cdots & x_{Rn_k}\mat{e}_{i_R}
            \end{bmatrix},
        \end{equation*}
        and thus the structure of $\spl_2(\core{\dot{P}}{k})$ will be
        \begin{equation*}
            \begin{bmatrix}
                x_{11}\mat{e}_{i_1} & x_{21}\mat{e}_{i_2} & \cdots & x_{R1}\mat{e}_{i_R} \\
                x_{12}\mat{e}_{i_1} & x_{22}\mat{e}_{i_2} & \cdots & x_{R2}\mat{e}_{i_R} \\
                \vdots & \vdots & \ddots & \vdots \\
                x_{1n_k}\mat{e}_{i_1} & x_{2n_k}\mat{e}_{i_2} & \cdots & x_{Rn_k}\mat{e}_{i_R}
            \end{bmatrix}.
        \end{equation*}
        From this it follows that $\spl_2(\core{\dot{P}}{k})$ is a quasi-permutation matrix. The same line of reasoning can be used to prove the theorem for $k > p$.
    \end{proof}
    
    Now, we consider how to perform Deparallelisation for a quasi-permutation matrix. Our aim is to express a matrix $\mat{M}$ as the product of two smaller matrices: $\mat{M=NT}$. For a quasi-permutation matrix, we need to remove the duplicate columns in $\mat{M}$. As shown in \cref{cor:qp}, the result $\mat{T}$ itself is a quasi-permutation matrix. Therefore, $\mat{N=I}$ and $\mat{T=M}$, where $\mat{I}$ is an identity matrix, can be regarded as the result of performing Deparallelisation on a quasi-permutation matrix, except that the duplicate columns in $\mat{M}$ have not yet been removed. What remains to be done is the removal of zero rows of $\mat{T}$ and the corresponding columns in $\mat{N}$. This is described as \cref{alg:dedup}.
    \begin{algorithm}[!htbp]
        \caption{Deparallelisation for a quasi-permutation matrix}
        \label{alg:dedup}
        \begin{algorithmic}[1]
            \setcounter{ALC@unique}{0}
            \REQUIRE A quasi-permutation matrix $\mat{M}\inR{n_1\times n_2}$.
            \ENSURE Matrices $\mat{N}\inR{n_1\times\beta}$, $\mat{T}\inR{\beta\times n_2}$ so that $\mat{M}=\mat{N}\mat{T}$, and $\mat{N}$ includes  nonduplicate columns of $\mat{M}$.
            \STATE $\beta \coloneqq 0$.
            \STATE Let $\mat{N}\inR{n_1\times\beta}$, $\mat{T}\inR{\beta\times n_2}$ be two  dynamically resized transfer matrices.
            \FOR{$i= 1, 2, \cdots, n_1$}
            \IF{$\mat{M}_{i,:}$ is not a zero row}   
            \STATE $\beta \coloneqq \beta+1$.  \COMMENT{extend matrix $\mat{N}$ and $\mat{T}$}
            \STATE $\mat{T}_{\beta,:} \coloneqq \mat{M}_{i,:}$.
            \STATE Set $\mat{N}_{:,\beta}$ a zero column except $\mat{N}_{i,\beta}=1$.
            \ENDIF
            \ENDFOR
            \STATE Return $\mat{N}$ and $\mat{T}$.
        \end{algorithmic}
    \end{algorithm}
    
    For a quasi-permutation matrix, each column can be represented by the position of 1 in it. Thus, \cref{alg:dedup} has a time complexity of $O(n_1+n_2)$, where $n_1$ and $n_2$ are the dimensions of $\mat{M}$. It can be executed much more efficiently than a general Deparallelisation algorithm. From \cref{alg:dedup} we can also observe, that the resulting matrix size $\beta$ must be no more than $n_1$, even if $n_2\gg n_1$.

    According to \cref{thm:dedup} and the above analysis, with \cref{alg:dedup_rounding} the TT ranks will be reduced to $\tilde{\mat{r}}$ satisfying the following upper bounds
    
    \begin{align}\label{equ:upper}
        \tilde{r}_k\le\bar{r}_k=\left\{\begin{array}{ll}
            \min\left(R,\prod_{i=1}^{k}n_i\right)&\text{if }1\le k<p,\\[2ex]
            \min\left(R,\prod_{i=k+1}^{d}n_i\right)&\text{if }p\le k<d.
        \end{array}\right.
    \end{align}
    where $R$ is the number of nonzero $p$-fibers in the original tensor $\ten{A}$.
    
    \subsection{More efficient TT-rounding and the FastTT algorithm}
    
    \Cref{alg:dedup_rounding} can already provide rank-reduced TT for sparse tensors while keeping the sparsity and with no precision loss. However, if lower ranks are desired, we can further apply TT-rounding on the TT. This could be useful for applications which do not care about the sparsity. Based on the property of TT obtained in \cref{alg:dedup_rounding}, we modified the original \cref{alg:ttro} in order to make it more efficient, described as \cref{alg:tt_rounding}. Instead of performing a right-to-left QR-sweep and then a left-to-right SVD-sweep, we perform 2 middle-to-edge SVD-sweep from core $p$. The QR-sweep in our algorithm is proved to be extremely fast due to the orthogonality obtained in \cref{alg:dedup_rounding} and former SVD, and that is why our algorithm is more efficient. The truncation parameters in \cref{alg:tt_rounding} satisfy
    \begin{align}\label{equ:static}
        \delta_k\coloneqq\frac{\varepsilon}{\sqrt{p-1}+\sqrt{d-p}}\|\ten{A}\|_F,~k=1\ldots d-1.
    \end{align}
    
    \begin{algorithm}[!htbp]
        \caption{More efficient TT-rounding for the sparse TT conversion}
        \label{alg:tt_rounding}
        \begin{algorithmic}[1]
            \setcounter{ALC@unique}{0}
            \REQUIRE Cores $\core{G}{1},\ldots,\core{G}{d}$ of the TT-format tensor $\ten{A}$ with TT-ranks $r_1,\ldots,r_{d-1}$, desired accuracy tolerance $\varepsilon$, an integer $p$ ($1\le p \le d$).
            \ENSURE Cores $\core{G}{1},\ldots,\core{G}{d}$ of the TT-approximation $\ten{B}$ to $\ten{A}$ in the TT-format with TT-ranks $r_1,\ldots,r_{d-1}$. The computed approximation satisfies
            \begin{equation*}
                ||\ten{A}-\ten{B}||_F\le\varepsilon||\ten{A}||_F.
            \end{equation*}
            \STATE Set truncation parameters according to \cref{equ:static}.
            \STATE \textit{(First SVD-sweep)}
            \FOR{$k=p,\ldots,d-1$}
            \STATE $\mat{G}\coloneqq$ $\spl_2(\core{G}{k})$.\label{line:3}
            \STATE Compute $\delta_k$-truncated SVD: $\mat{G}=\mat{U\Sigma V}^T+\mat{E}_k$, $\|\mat{E}_k\|_F\le\delta_k$,\\ $r_k\coloneqq\rank_{\delta_k}(\mat{G})$.\label{line:4}
            \STATE $\core{G}{k}\coloneqq\reshape(\mat{U},[r_{k-1},n_k,r_k])$.
            \STATE $\core{G}{k+1}\coloneqq\core{G}{k+1}\times_1(\mat{V}\mat\Sigma)$.\label{line:merge}
            \ENDFOR
            \STATE \textit{(QR-sweep)}
            \FOR{$k=d,\ldots,p+1$}\label{line:qrb}
            \STATE $\mat{G}\coloneqq$ $\spl_1^T(\core{G}{k})$.
            \STATE Compute economic QR decomposition: $\mat{G}=\mat{Q}\mat{R}$.
            \STATE $\core{G}{k}\coloneqq\reshape(\mat{Q}^T,[r_{k-1},n_k,r_k])$.
            \STATE $\core{G}{k-1}\coloneqq\core{G}{k-1}\times_3\mat{R}^T$.
            \ENDFOR\label{line:qre}
            \STATE \textit{(Second SVD-sweep)}
            \FOR{$k=p,\ldots,2$}
            \STATE $\mat{G}\coloneqq$ $\spl_1^T(\core{G}{k})$.\label{line:15}
            \STATE Compute $\delta_{k-1}$-truncated SVD: $\mat{G}=\mat{U\Sigma V}^T+\mat{E}_{k-1}$, $\|\mat{E}_{k-1}\|_F\le\delta_{k-1}$,\\ $r_{k-1}\coloneqq\rank_{\delta_{k-1}}(\mat{G})$.
            \STATE $\core{G}{k}\coloneqq\reshape(\mat{U}^T,[r_{k-1},n_k,r_k])$.
            \STATE $\core{G}{k-1}\coloneqq\core{G}{k-1}\times_3(\mat{V}\mat\Sigma)$.\label{line:merge2}
            \ENDFOR
            \STATE Return $\core{G}{1},\ldots,\core{G}{d}$ as cores of $\ten{B}$.
        \end{algorithmic}
    \end{algorithm}
    
    The  correctness of \cref{alg:tt_rounding} is explained as follows.
    
    \begin{lemma}\label{lem:qp}
        A quasi-permutation matrix with no duplicate columns is an orthonormal matrix. 
    \end{lemma}
    
    \Cref{lem:qp} can be easily proved based by the \cref{def:qp} and the definition of an orthonormal matrix. 
    
    
    In Steps 9 and 15 of \cref{alg:dedup_rounding}, the duplicate columns in the input quasi-permutation matrix (according to \cref{thm:dedup}) are removed. Then, based on \cref{lem:qp}, we have the following statement. 
    
    \begin{corollary}\label{cor:lo}
        Supposes the TT cores $\core{G}{k}$ are obtained with \cref{alg:dedup_rounding}. Then the matrices $\spl_2(\core{G}{k}), k<p$ and $\spl_1^T(\core{G}{k}), k>p$ are all orthonormal matrices.
    \end{corollary}
    
    \begin{lemma}\label{lem:lo_product}
        Suppose $\core{U}{i}, i = 1,\ldots,d$ are the cores of a tensor train. If matrix $\spl_2(\core{U}{i})$ is an orthonormal matrix for all $i = 1,\ldots,k$ ($1\le k\le d$), then the matrix $\spl_j(\core{U}{1}\circ\cdots\circ\core{U}{j})$ is an orthonormal matrix for all $j=1,\ldots,k$.
    \end{lemma}
    
    The proof of \cref{lem:lo_product} can be found in \cite[Appendix B]{wang2018tensor}. We can now derive the following theorem.
    
    \begin{theorem}\label{thm:tt_rounding}
        (Correctness of \cref{alg:tt_rounding}) The approximation $\ten{B}$ obtained in \cref{alg:tt_rounding} always satisfies $\|\ten{A}-\ten{B}\|_F\le\varepsilon\|\ten{A}\|_F$.
    \end{theorem}
    
    \begin{proof}
        For simplicity, we let
        \begin{itemize}
            \item $\ten{C}$ denote the TT-format tensor after first SVD-sweep,
            \item $\ten{D}$ denote the TT-format tensor before second SVD-sweep,
            \item $\ten{C}_i$ denote the TT-format tensor after the $k=i$ iteration in first SVD-sweep,
            \item $\core{A}{i\dots j}$ denote the contraction of $i$-th to $j$-th core $\core{G}{i}\circ\cdots\circ\core{G}{j}$ of a TT-format tensor $\ten{A}$.
        \end{itemize}
        
        It is obvious that $\ten{C}=\ten{D}$, hence
        \begin{equation*}
            \|\ten{A}-\ten{B}\|_F\le\|\ten{A}-\ten{C}\|_F+\|\ten{D}-\ten{B}\|_F.
        \end{equation*}
        
        Based on the observation that $\core{A}{1\dots p-1}=\core{C}{1\dots p-1}$, we let $\ten{A}=\core{A}{1\dots p-1}\circ\core{A}{p\dots d}$ and $\ten{C}=\core{A}{1\dots p-1}\circ\core{C}{p\dots d}$. According to \cref{cor:lo,lem:lo_product}, $\spl_{p-1}(\core{A}{1\dots p-1})$ is an orthonormal matrix. Thus
        \begin{equation*}
            \|\ten{A}-\ten{C}\|_F
            =\|\core{A}{p\dots d}-\core{C}{p\dots d}\|_F.
        \end{equation*}
        
        Let us concentrate on the first iteration ($k=p$) of first SVD-sweep. The truncated-SVD can be rewritten as $\core{A}{p}=\core{C}{p}_1\times_3\Sigma\mat{V}^T+\ten{E}_p$, where $\|\ten{E}_p\|_F\le\delta_p$ and $\core{C}{p}_1\circ_3^3\ten{E}_p=\mat{0}$. From \cref{line:merge} we know $\core{C}{p+1\dots d}_1=\core{A}{p+1\dots d}\times_1\mat{V\Sigma}$. Thus,
        \begin{align*}
            \|\core{A}{p\dots d}-\core{C}{p\dots d}\|_F^2
            &=\|\core{A}{p}\circ\core{A}{p+1\dots d}-\core{C}{p}_1\circ\core{C}{p+1\dots d}\|_F^2 \\
            &=\|(\core{C}{p}_1\times_3\Sigma\mat{V}^T+\ten{E}_p)\circ\core{A}{p+1\dots d}-\core{C}{p}_1\circ\core{C}{p+1\dots d}\|_F^2 \\
            &=\|\core{C}{p}_1\circ\core{C}{p+1\dots d}_1+\ten{E}_p\circ\core{A}{p+1\dots d}-\core{C}{p}_1\circ\core{C}{p+1\dots d}\|_F^2 \\
            &=\|\ten{E}_p\circ\core{A}{p+1\dots d}\|_F^2+\|\core{C}{p}_1\circ(\core{C}{p+1\dots d}_1-\core{C}{p+1\dots d})\|_F^2 \\
            &=\delta_p^2+\|\core{C}{p+1\dots d}_1-\core{C}{p+1\dots d}\|_F^2
        \end{align*}
        
        Proceeding by induction, we have
        \begin{equation*}
            \|\ten{A}-\ten{C}\|_F^2=\sum_{k=p}^{d-1}\delta_k^2.
        \end{equation*}
        
        Similarly we have
        \begin{equation*}
            \|\ten{D}-\ten{B}\|_F^2=\sum_{k=2}^{p}\delta_{k-1}^2.
        \end{equation*}
        
        According to \cref{equ:static},
        \begin{align}\label{equ:error}
            \|\ten{A}-\ten{B}\|_F\le\sqrt{\sum_{k=p}^{d-1}\delta_k^2}+\sqrt{\sum_{k=2}^{p}\delta_{k-1}^2}\le\varepsilon\|\ten{A}\|_F.
        \end{align}
    \end{proof}
    
    Now, we are ready to describe the whole algorithm for the conversion of a sparse tensor into a TT, presented as \cref{alg:fastTT}.
    \begin{algorithm}[htb]
        \caption{Tensor train decomposition of sparse tensor (FastTT)}
        \label{alg:fastTT}
        \begin{algorithmic}[1]
            \setcounter{ALC@unique}{0}
            \REQUIRE A sparse tensor $\ten{A}\inR{n_1\times n_2\times\ldots\times n_d}$, desired accuracy tolerance $\varepsilon$, an integer $p$ ($1\le p \le d$).
            \ENSURE Cores $\core{G}{1},\ldots,\core{G}{d}$ of the TT-approximation $\ten{B}$ to $\ten{A}$ in the TT-format with TT-ranks $r_k$. The computed approximation satisfies
            \begin{equation*}
                ||\ten{A}-\ten{B}||_F\le\varepsilon||\ten{A}||_F.
            \end{equation*}
            \STATE Obtain $\ten{B}$ in the TT-format with cores $\core{G}{1},\ldots,\core{G}{d}$ by \cref{alg:dedup_rounding}, where \emph{Depar} is implemented as \cref{alg:depar}.
            \STATE Reduce the TT ranks of $\ten{B}$ with \cref{alg:tt_rounding}.
            \STATE Return the TT-approximation $\ten{B}$.
        \end{algorithmic}
    \end{algorithm}
    
    It should be pointed out that if the accuracy tolerance $\varepsilon$ is set to 0\footnote{In practice, $\varepsilon$ is usually set to a small value like $10^{-14}$ due to the inevitable round-off error.}, the obtained TT ranks with \cref{alg:fastTT} will be maximal and equal to the TT-ranks obtained from the TT-SVD algorithm.
    We take $p=1$ as an example to discuss the TT-rank $r_1$ case. In the TT-SVD algorithm, $r_1$ is obtained by computing the SVD of $\mat{C}=\spl_1(\ten{A})$. $\ten{A}$ can also represented as the contraction of the TT cores obtained by \cref{alg:dedup_rounding}.
    \begin{equation*}
        \ten{A}=\core{G}{1}\circ\core{G}{2}\circ\cdots\circ\core{G}{d}.
    \end{equation*}
    Thus,
    \begin{equation*}
        \mat{C}=\spl_2(\core{G}{1})\spl_1(\core{G}{2}\circ\cdots\circ\core{G}{d}).
    \end{equation*}
    According to \cref{cor:lo,lem:lo_product}, $\mat{L}=\spl_1^T(\core{G}{2}\cdots\core{G}{d})$ is an orthonormal matrix, i.e. $\mat{L}^T\mat{L}=\mat{I}$. 
    \begin{equation*}
        \mat{C}\mat{C}^T=\spl_2(\core{G}{1})\mat{L}^T\mat{L}\spl_2^T(\core{G}{1})=\spl_2(\core{G}{1})\spl_2^T(\core{G}{1}).
    \end{equation*}
    This means matrix $\spl_2(\core{G}{1})$ has the same singular values as $\mat{C}$. For \cref{alg:TTSVD},  $r_1$ equals $\mathrm{rank}_{\delta}(\mat{C})$, while  the $r_1$ obtained with \cref{alg:fastTT} is $\mathrm{rank}_{\delta_1}(\spl_2(\core{G}{1}))$ (see \cref{line:3,line:4} of \cref{alg:tt_rounding}). These numerical ranks are therefore equal when $\delta = \delta_1$. Similar results for the other TT ranks and for $p\ne 1$ can be derived. 
    
    For a sparse tensor the runtime of \cref{alg:fastTT} may be smaller than the TT-SVD algorithm, as the SVD is performed on smaller matrices.
    
    \subsection{Fixed-rank TT approximations and matrices in TT-format}
    
    Sometimes we need a TT approximation of a tensor with given TT-ranks.
    We can slightly modify \cref{alg:tt_rounding} to fit this scenario. Specifically,
    the desired accuracy tolerance $\varepsilon$ is not needed and thus substituted with the desired TT-ranks. The truncation parameters $\delta_i$ will not be computed either. In the truncated SVD computation we simply truncate the matrices with the given ranks instead of truncating them according to the accuracy tolerance. This technique could be useful in applications like tensor completion \cite{ko2018fast}.
    
    Some other applications require matrix-vector multiplications, which are convenient if both the matrix and the vector are in TT-format (as shown in \cref{fig:mpo}). A vector $\mat{v}\inR{N}$ can be transformed into TT-format if we first reshape it into a tensor $\ten{V}\inR{n_1\times\cdots\times n_d}$, where $N=n_1\cdots n_d$, and then decompose it into a TT. A ``matrix in TT-format''\cite[pp. 2311-2313]{oseledets2011tensor}, also known as a matrix product operator (MPO), is similar but more complicated. The elements of matrix $\mat{M}\inR{M\times N}$ are rearranged into a tensor $\ten{M}\inR{m_1\times n_1\times\cdots\times m_d\times n_d}$, where $M=m_1\cdots m_d, N=n_1\cdots n_d$. The cores $\core{M}{i}(i=1,\ldots,d)$ of the MPO satisfy
    \begin{equation*}
        \ten{M}(i_1,j_1,\ldots,i_d,j_d)=\core{M}{1}(:,i_1,j_1,:)\cdots\core{M}{d}(:,i_d,j_d,:),
    \end{equation*}
    where $\core{M}{i}\inR{r_{i-1}\times m_i\times n_i\times r_i}(i=1,\ldots,d),r_0=r_d=1$. The matrix-to-MPO algorithm is basically computing a TT-decomposition of the $d$-way tensor $\ten{M}'\inR{m_1n_1\times\cdots\times m_dn_d}$, along with a few necessary reshapings, which can also be done with \cref{alg:fastTT}.
    
    \begin{figure}[!htbp]
        \centering
        \resizebox{160px}{!}{
            \begin{tikzpicture}[text height=1.5ex, text depth=0.25ex, auto, node distance=1.5cm, ampersand replacement=\&]
                \tikzstyle{tensor} = [circle, thick, inner sep=0pt, minimum size=1cm, draw=black]
                \tikzstyle{null} = [circle, minimum size=0cm]
                \matrix[column sep=0.75cm, row sep=0.75cm] {
                    \node[tensor] (m1) {$\core{M}{1}$}; \&
                    \node[tensor] (m2) {$\core{M}{2}$}; \&
                    \node[tensor] (m3) {$\core{M}{3}$}; \&
                    \\
                    \node[tensor] (v1) {$\core{V}{1}$}; \&
                    \node[tensor] (v2) {$\core{V}{2}$}; \&
                    \node[tensor] (v3) {$\core{V}{3}$}; \&
                    \\
                };
                \node[null] (m0) [left of=m1] {};
                \node[null] (n1) [above of=m1] {};
                \node[null] (n2) [above of=m2] {};
                \node[null] (n3) [above of=m3] {};
                \node[null] (m4) [right of=m3] {};
                \node[null] (v0) [left of=v1] {};
                \node[null] (v4) [right of=v3] {};
                \path[every node/.style={above}]
                (m0) edge[dashed] node{$r_0$} (m1)
                (m1) edge node{$r_1$} (m2)
                (m2) edge node{$r_2$} (m3)
                (m3) edge[dashed] node{$r_3$} (m4)
                (v0) edge[dashed] node{$r'_0$} (v1)
                (v1) edge node{$r'_1$} (v2)
                (v2) edge node{$r'_2$} (v3)
                (v3) edge[dashed] node{$r'_3$} (v4)
                ;
                \path[every node/.style={left}]
                (m1) edge node{$m_1$} (n1)
                (m2) edge node{$m_2$} (n2)
                (m3) edge node{$m_3$} (n3)
                (m1) edge node{$n_1$} (v1)
                (m2) edge node{$n_2$} (v2)
                (m3) edge node{$n_3$} (v3)
                ;
            \end{tikzpicture}
        }
        \caption{Diagram of matrix-vector-multiplication in the TT-format.}
        \label{fig:mpo}
    \end{figure}
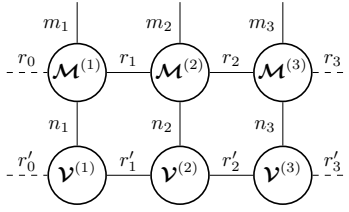
    
    \subsection{A dynamic method to choose the truncation parameters}\label{sec:dynamic}
    
    The actual relative error of the truncated SVD is usually not very close to the truncation parameter $\delta_k$, which implies that if the truncation parameters are set statically at the beginning with \cref{equ:static}, some of the desired accuracy tolerant $\varepsilon$ will be ``wasted''. The main idea of the dynamic method is to compute the truncation parameters dynamically to make use of those ``wastes''. One of the possible approaches is shown in \cref{alg:dynamic}. In each step of the truncated SVD, an expected error is calculated with the current ``total error remainder'' and used as the truncation parameter. After each truncated SVD, the ``total error remainder'' is decreased according to the actual error.
    Such an adaptive approach to setting the truncation parameters can lead to lower TT ranks while keeping the relative error smaller than $\varepsilon$.
    
    \begin{algorithm}[!htbp]
        \caption{The revised TT-rounding for the sparse TT conversion}
        \label{alg:dynamic}
        \begin{algorithmic}[1]
            \setcounter{ALC@unique}{0}
            \REQUIRE \textit{(same as \cref{alg:tt_rounding})}
            \ENSURE \textit{(same as \cref{alg:tt_rounding})}
            \STATE $\delta_\text{right}\coloneqq\frac{\sqrt{d-p}}{\sqrt{d-p}+\sqrt{p-1}}\varepsilon\|\ten{A}\|_F,~
            \delta_\text{left}\coloneqq\frac{\sqrt{p-1}}{\sqrt{d-p}+\sqrt{p-1}}\varepsilon\|\ten{A}\|_F$.
            \FOR{$k=p,\ldots,d-1$}\label{line:loop1_start}
            \STATE $\delta_k\coloneqq\frac{\delta_\text{right}}{\sqrt{d-k}}$.
            \STATE \hyperref[line:3]{Steps~\ref*{line:3}}-\ref{line:merge} of \cref{alg:tt_rounding}.
            \STATE $\delta_\text{right}\coloneqq\sqrt{\delta_\text{right}^2-\|\mat{E}_k\|_F^2}$.
            \ENDFOR\label{line:loop1_end}
            \STATE \hyperref[line:qrb]{Steps~\ref*{line:qrb}}-\ref{line:qre} of \cref{alg:tt_rounding}.
            \FOR{$k=p,\ldots,2$}\label{line:loop2_start}
            \STATE $\delta_{k-1}\coloneqq\frac{\delta_\text{left}}{\sqrt{k-1}}$.
            \STATE \hyperref[line:15]{Steps~\ref*{line:15}}-\ref{line:merge2} of \cref{alg:tt_rounding}.
            \STATE $\delta_\text{left}\coloneqq\sqrt{\delta_\text{left}^2-\|\mat{E}_{k-1}\|_F^2}$.
            \ENDFOR\label{line:loop2_end}
            \STATE Return $\core{G}{1},\ldots,\core{G}{d}$ as cores of $\ten{B}$.
        \end{algorithmic}
    \end{algorithm}
    
    \begin{theorem}\label{thm:dyn}
        (Correctness of \cref{alg:dynamic}) The approximation $\ten{B}$ obtained in \cref{alg:dynamic} always satisfies $\|\ten{A}-\ten{B}\|_F\le\varepsilon\|\ten{A}\|_F$.
    \end{theorem}
    
    The proof of \cref{thm:dyn} is provided in \cref{proof:dyn}.
    
    \subsection{Complexity Analysis}
    
    Finding nonzero $p$-fibers can be accelerated by employing balanced binary search trees or hash tables, while parallel-vector rounding will be efficient if \emph{Depar} is implemented as \cref{alg:dedup}. Notice that, there is no floating point operation in these procedures. Therefore, the time complexity of \cref{alg:fastTT} mainly depends on \cref{alg:tt_rounding}, where the cost of SVD is of major concern. The FLOP count $f_{\mathrm{fasttt}}$ can thus be estimated as
    \begin{align}\label{equ:flop}
        f_{\mathrm{fasttt}}\approx f_{\mathrm{SVD}}(\tilde{r}_{p-1}n_p, \tilde{r}_p)
        +\sum_{i=p+1}^{d-1}f_{\mathrm{SVD}}(r_{i-1}n_i,\tilde{r}_i)
        +\sum_{i=2}^{p}f_{\mathrm{SVD}}(\tilde{r}_{i-1}, n_ir_i),
    \end{align}
    where $\{\tilde{r}_k\}$ and $\{{r}_k\}$ are the TT-ranks before and after executing \cref{alg:tt_rounding}. According to \cref{equ:upper}, where the upper bound of $\tilde{r}_k$, i.e., $\bar{r}_k$, is given, we can estimate the upper bound of the FLOP count before any actual computation. With this estimation, $p$ can be automatically selected as described in \cref{alg:select_p}. In \cref{line:tilde_r}, $\{\tilde{r}_k\}$ can be obtained alternatively by actually executing \textit{Depar} for a more precise estimation since it will not take much time after all.
    
    \begin{algorithm}[!htbp]
        \caption{Automatically select $p$}
        \label{alg:select_p}
        \begin{algorithmic}[1]
            \setcounter{ALC@unique}{0}
            \REQUIRE A sparse tensor $\ten{A}\inR{n_1\times n_2\times\ldots\times n_d}$.
            \ENSURE Selected $\bar{p}$ for best estimated performance in \cref{alg:fastTT}.
            \FOR{$p=1,\ldots,d$}
            \STATE $R\coloneqq$ the number of nonzero $p$-subvectors of $\ten{A}$.
            \STATE $\{\tilde{r}_k\}\coloneqq\{\bar{r}_k\}$ given in \cref{equ:upper}.\label{line:tilde_r}
            \STATE $\{r_k\}\coloneqq\{\tilde{r}_k\}$, or specified by users.
            \STATE $f_p\coloneqq$ the estimated FLOP count in \cref{equ:flop}.
            \ENDFOR
            \STATE Return $\bar{p}=\argmin{p}f_p$.
        \end{algorithmic}
    \end{algorithm}
    
    For a more intuitive view of the time complexity, we analyze the FLOP counts for an example from \hyperref[sec:image]{Section~\ref*{sec:image}}. Suppose we are computing a fixed rank-10 TT-approximation of a sparse $7$-way tensor $\ten{A}\inR{10\times20\times20\times10\times15\times20\times3}$ with density \mbox{$\sigma=0.001$}. According to \cref{equ:ttsvd}, the approximate FLOP count of TT-SVD is
    \begin{align*}
        f_{\text{TTSVD}}\approx& f_{\svd}(10,20\times20\times10\times15\times20\times3)+f_{\svd}(20r,20\times10\times15\times20\times3)\\
        &+f_{\svd}(20r,10\times15\times20\times3)+\cdots+f_{\svd}(20r,3)\\
        \approx&(3.6\times10^8+7.6r^2\times10^7)C_{\svd}\\
        \approx&(8\times10^9)C_{\svd}.
    \end{align*}
    As for $f_\text{fasttt}$, we let $p=7$. Since the elements of $\ten{A}$ are grouped in triples stored in the last dimension, the number of nonzero $7$-subvectors $R$ satisfies $R\le\operatorname{nnz}(\ten{A})/3=12000$, which means $\{\bar{r}_k\}$ given in \cref{equ:upper} is no more than $\{10,200,4000,12000,12000,12000\}$. According to \cref{equ:flop}, we have
    \begin{align*}
        f_{\text{fasttt}}\approx& f_{\svd}(3,12000)+f_{\svd}(20r,12000)+f_{\svd}(15r,12000)\\
        &+f_{\svd}(10r,4000)+f_{\svd}(20r,200)\\
        \approx&(1.08\times10^5+8r^2\times10^6)C_{\svd}\\
        \approx&(8\times10^8)C_{\svd}.
    \end{align*}
    
    In this case, \cref{alg:fastTT} is about 10X faster than TT-SVD. The actual speedup will be a bit lower due to the uncounted operations such as those in \cref{alg:dedup_rounding}. If we increase the density $\sigma$ to 0.01, $f_{\text{TTSVD}}$ will remain the same and $f_{\text{fasttt}}$ will change into
    \begin{align*}
        f_{\text{fasttt}}\approx& f_{\svd}(3,120000)+f_{\svd}(20r,120000)+f_{\svd}(15r,40000)\\
        &+f_{\svd}(10r,4000)+f_{\svd}(20r,200)\\
        \approx&(1.08\times10^6+5.7r^2\times10^7)C_{\svd}\\
        \approx&(5.7\times10^9)C_{\svd},
    \end{align*}
    and the speedup drops to 1.4. The actual speedup will be a bit higher because we overestimate $\{\tilde{r}_k\}$ and the uncounted operations become insignificant with the increasing SVD cost. For matrix-to-MPO applications, the main operation is the TT-decomposition of $\ten{M}'\inR{m_1n_1\times\cdots\times m_dn_d}$, hence \cref{equ:flop} also works here. We just need to replace $n_i$ with $m_in_i$.

    We now look at the memory cost of the FastTT algorithm. Before the TT-rounding step, all data are stored in a sparse format. Therefore, the extra memory cost occurs in the TT-rounding which is similar to the TT-SVD algorithm and is of similar size to the cores in the obtained TT.
    
    \section{Numerical Experiments}
    
    In order to provide empirical proof of the performance of the developed FastTT algorithm, we conduct several numerical experiments. Since \cref{alg:dedup_rounding} is not friendly to MATLAB, FastTT is implemented in C++ based on the \textit{xerus} C++ toolbox \cite{xerus} and with Intel Math Kernel Library\footnote{\url{https://software.intel.com/en-us/mkl}}. The \textit{xerus} library also contains implementations of TT-SVD and randomized TT-SVD (rTTSVD) \cite{huber2017randomized}. As no C++ implementation of the TT-cross method~\cite{oseledets2010tt} is available, we use TT-cross from \textit{TT-toolbox}\footnote{\url{https://github.com/oseledets/TT-Toolbox}} in MATLAB. All experiments are carried out on a x86-64 Linux server with 32 CPU cores and 512G RAM. The desired accuracy tolerance $\varepsilon$ of TT-SVD, TT-cross and our FastTT algorithms is $10^{-14}$, unless otherwise stated. The oversampling parameter of rTTSVD algorithm is set to 10. In all experiments, the CPU time is reported.
    
    \subsection{Image/video inpainting}\label{sec:image}
    
    Applications like tensor completion \cite{ko2018fast} require a fixed-rank TT-approximate of the given tensors. The tensors used in this section are a large color image \textit{Dolphin}\footnote{\url{http://absfreepic.com/absolutely_free_photos/original_photos/dolphin-4000x3000_21859.jpg}} which has been reshaped into a $10\times20\times20\times10\times15\times20\times3$ tensor and a color video \textit{Mariano Rivera Ultimate Career Highlights}\footnote{\url{https://www.youtube.com/watch?v=UPtDJuJMyhc}} which has been reshaped into a $20\times18\times20\times32\times12\times12\times3$ tensor. Most pixels of the image/video are not observed and are regarded as zeros whereas the observed pixels are chosen randomly. The observation ratio $\sigma$ is the ratio of observed pixels to the total number of pixels. \Cref{tab:exp3} shows the results for different specified TT-ranks and observation ratios.
    
    \begin{table}[htbp]
        \centering
        \caption{Experimental results on an image and a video with different observation ratios and preset TT-ranks.}
        \label{tab:exp3}
        \begin{tabular}{@{~}c@{~}c@{~}c@{~}c@{~}c@{~}c@{~}c@{~}c@{~}c@{~}}
            \toprule
            \multirow{2}*{data} & \multirow{2}*{TT-rank} & \multirow{2}*{$\sigma$} & \multicolumn{4}{c}{time (s)} & & \multirow{2}*{speedup} \\
            \cmidrule{4-7}
            & & & TT-SVD & TT-cross & rTTSVD & FastTT & &  \\
            \midrule
            \multirow{6}*{image} & 10 & 0.001 & 20.2 & 20.2 & 24.1 & 3.43 & & 9.6X \\
            & 10 & 0.005 & 32.3 & 22.7 & 23.9 & 10.9 & & 3.0X \\
            & 10 & 0.01 & 32.8 & 23.7 & 26.0 & 14.2 & & 2.3X \\
            & 30 & 0.001 & 42.7 & 68.1 & 38.1 & 12.2 & & 3.5X \\
            & 30 & 0.005 & 42.9 & 70.9 & 33.8 & 20.5 & & 2.1X \\
            & 100 & 0.001 & 67.3 & 366 & 91.5 & 23.7 & & 2.8X \\
            \midrule
            \multirow{6}*{video} & 10 & 0.001 & 66.2 & 24.9 & 56.0 & 10.4 & & 6.4X \\
            & 10 & 0.005 & 66.6 & 30.3 & 60.5 & 26.2 & & 2.5X \\
            & 10 & 0.01 & 66.9 & 31.2 & 62.7 & 33.3 & & 2.0X \\
            & 30 & 0.001 & 103 & 122 & 108 & 26.5 & & 3.9X \\
            & 30 & 0.005 & 110 & 140 & 94.2 & 47.6 & &2.3X \\
            & 100 & 0.001 &232 & 1080 & 221 & 107 & &2.2X \\
            \bottomrule
        \end{tabular}
    \end{table}
    
    It can be seen from \cref{tab:exp3} that our algorithm can greatly speed up the calculation of a TT-approximation when the observation ratio is small. We have also tested the TT-cross algorithm and the rTTSVD algorithm which also speed up the calculation in some cases. However, the speedup of them are not as great as ours, and in cases where the preset TT-rank is high we observe that they are even slower than the TT-SVD algorithm. In addition, the quality of the TT-approximation calculated by the TT-cross/rTTSVD algorithm is not as good as ours. For example, in the image inpainting task where the TT-rank is 100 and the observation ratio $\sigma$ is 0.001, the mean square error (MSE) of both TT-SVD algorithm and our algorithm is 22.3, while the MSE of TT-cross and rTTSVD is 66.0 and 23.5, correspondingly.
    
    For each of the experiments the integer $p$ is selected automatically by the FLOP estimation in \cref{alg:select_p}. Now, we validate this FLOP estimation. For the parameters $\text{TT-rank}=100$, $\sigma=0.001$ in the image experiment we run \cref{alg:fastTT} several times while manually setting different integer $p$ and plot the CPU time for each $p$ along with the estimated FLOP count. The results are shown in \cref{fig:select_p}, where we can see that the trend of the two curves is basically consistent. The integer $p$ selected by \cref{alg:select_p} is $p=7$, with which the exact CPU time is only slightly more than the best selection at $p=6$. Although \Cref{alg:select_p} does not always produce the best $p$, it certainly avoids bad values like $p=3$ in this case.
    
    \begin{figure}[htbp]
        \centering
        \includegraphics[width=250px]{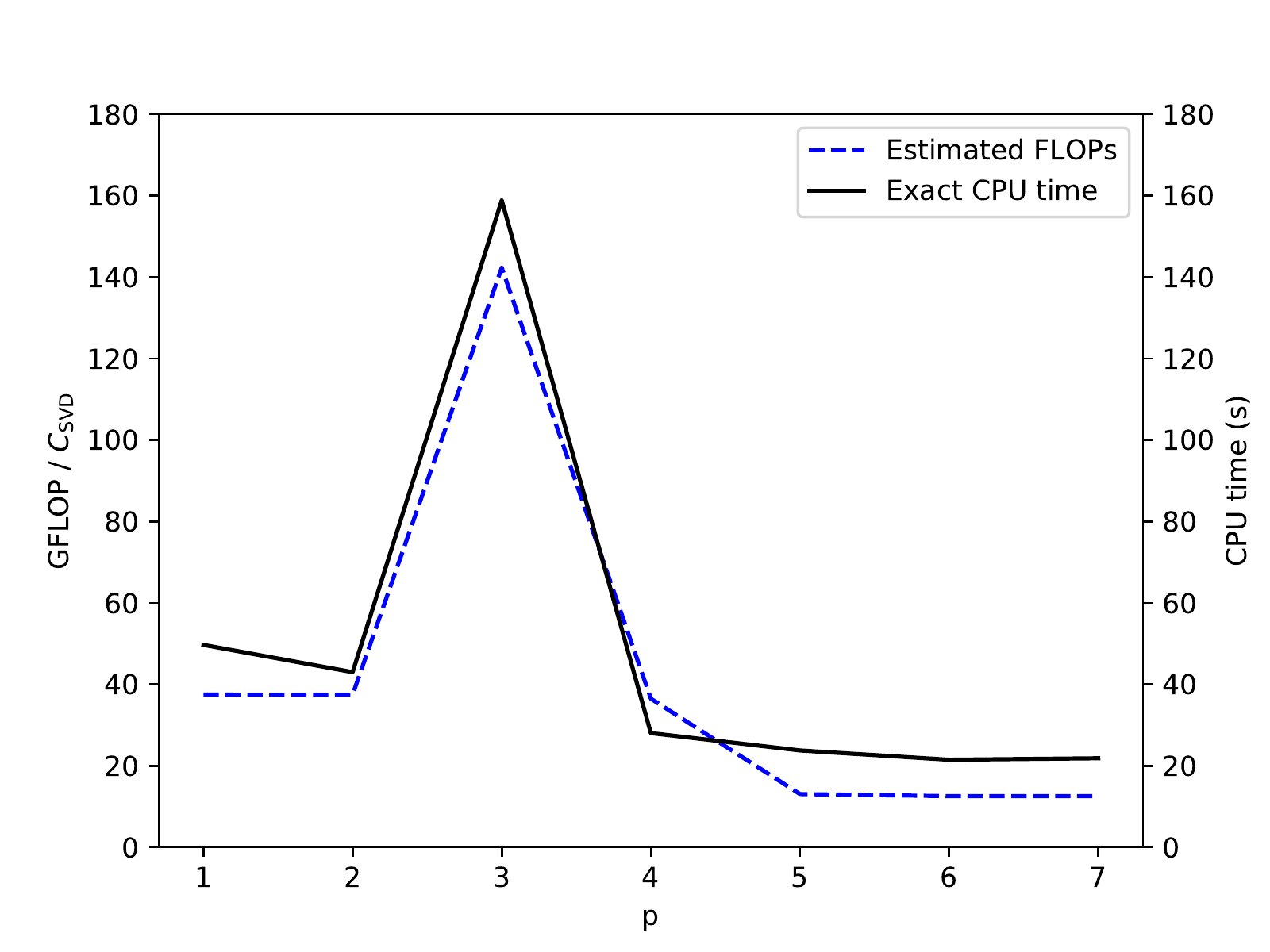}
        \caption{The CPU time and estimated FLOPs with \cref{equ:flop} of the FastTT algorithm for different $p$ values.}
        \label{fig:select_p}
    \end{figure}
    
    \subsection{Linear equation in finite difference method}
    
    The finite difference method (FDM) is widely used for solving partial differential equations, in which finite differences approximate the partial derivatives. With FDM, a linear equation system with sparse coefficient matrix is solved. We consider simulating a three-dimensional rectangular domain with FDM. The resulted linear equation system can be transformed into the matrix TT format (i.e. MPO) and then solved with an alternating least squares (ALS) method~\cite{oseledets2012solution} or AMEn \cite{dolgov2015corrected}.
    
    
    \begin{figure}[htbp]
        \centering
        \includegraphics[width=180px]{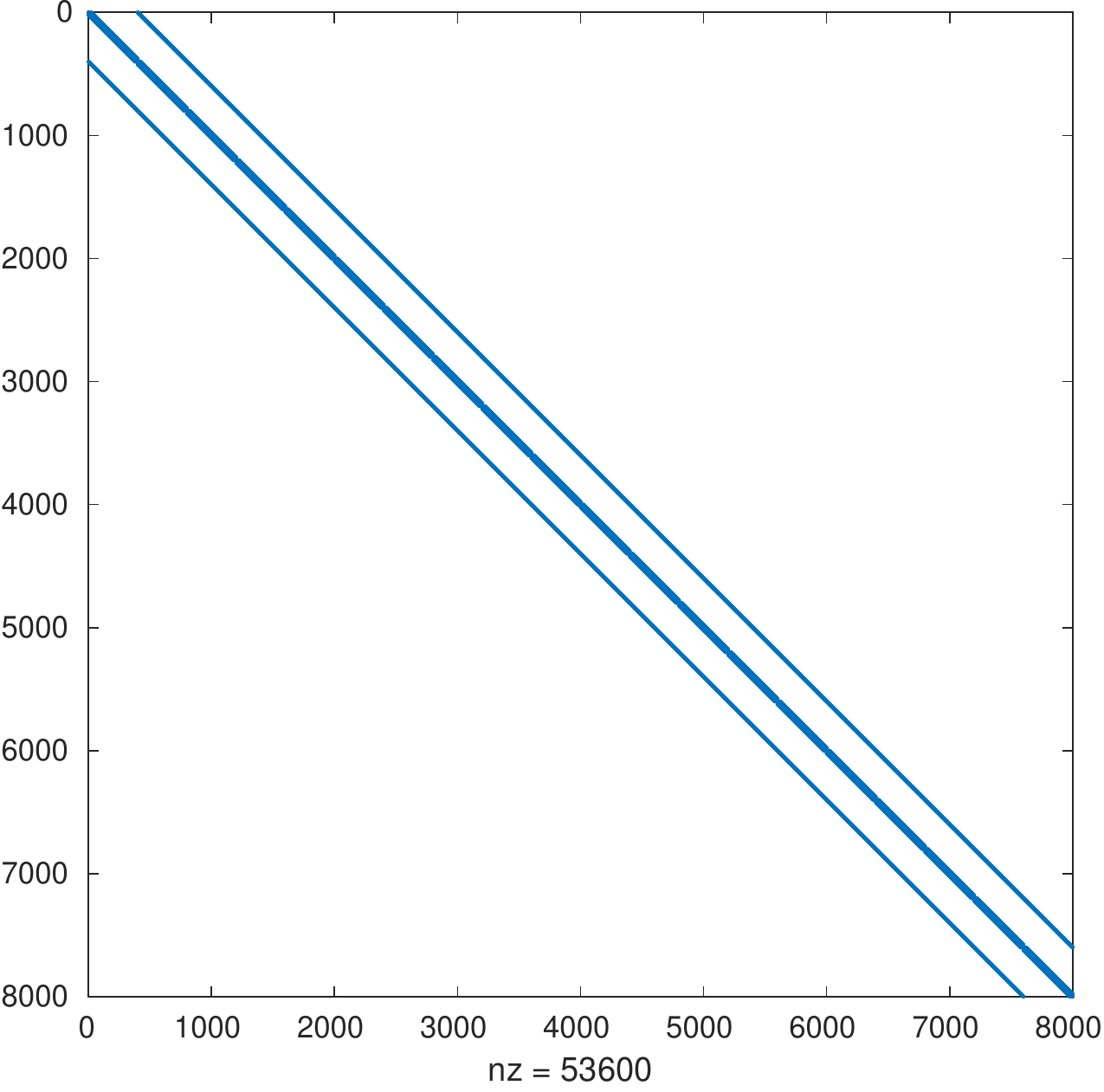}
        \caption{The sparsity pattern of the coefficient matrix for FDM with $20\times20\times20$ grids.}
        \label{fig:spy_fdm}
    \end{figure}
    
    For a domain partitioned into $n\times m\times k$ grids, FDM produces a coefficient matrix $\mat{A}\inR{N\times N}$, where $N=n\times m\times k$.
    For example, the sparsity pattern of the coefficient matrix $\mat{A}$ for FDM with $20\times20\times20$ grids is shown in \cref{fig:spy_fdm}.
    Naturally, the $\mat{A}$ matrix can be regarded as a 6-way tensor $\ten{A} \inR{n\times n\times m\times m\times k\times k}$, which is then converted into an MPO. 
    Since the tensor $\ten{A}$ is very sparse, replacing the TT-SVD with FastTT will speed up the procedure of computing its TT-decomposition. In this experiment we construct the coefficient matrix with different grid partition, while the coefficients either follow a particular pattern, or are randomly generated. The results for converting the matrix to an MPO
    are shown in \cref{tab:exp_fdm}.
    
    \begin{table}[htbp]
        \centering
        \caption{Experimental results on the coefficient matrices for the FDM with $n\times n\times n$ grids.}
        \label{tab:exp_fdm}
        \begin{threeparttable}
            \begin{tabular}{@{~}c@{~}c@{~}c@{~}c@{~}c@{~}c@{~}c@{~}}
                \toprule
                $n$ & coefficients & method & time(s) & speedup & $\varepsilon_{\text{actual}}$\tnote{1} & TT-ranks\tnote{2} \\
                \midrule
                \multirow{4}*{20} & \multirow{4}*{pattern} & TT-SVD & 43.6 & -- & $4.0\!\times\!10^{-16}$ & ${r}:2,2$ \\
                & & TT-cross & 7.96 & 5.5X & $2.0\!\times\!10^{-16}$ & ${r}:2,2$ \\
                & & rTTSVD & 1.29 & 34X & $1.2\!\times\!10^{-15}$ & ${r}:2,2$ \\
                & & FastTT & 0.788 & 55X & $9.6\!\times\!10^{-16}$ & $R\!:\!1920;~\tilde{{r}}\!:\!58,\!58;~{r}\!:\!2,\!2$ \\
                \midrule
                \multirow{4}*{30} & \multirow{4}*{pattern} & TT-SVD & 690 & -- & $2.0\!\times\!10^{-15}$ & ${r}:2,2$ \\
                & & TT-cross & 26.5 & 26X & $8.8\!\times\!10^{-16}$ & ${r}:2,2$ \\
                & & rTTSVD & 19.3 & 36X & $1.6\!\times\!10^{-15}$ & ${r}:2,2$ \\
                & & FastTT & 2.88 & 240X & $1.1\!\times\!10^{-15}$ & $R\!:\!4380;~\tilde{{r}}\!:\!88,\!88;~{r}\!:\!2,\!2$ \\
                \midrule
                \multirow{4}*{20} & \multirow{4}*{random} & TT-SVD & 53.4 & -- & $2.5\!\times\!10^{-15}$ & ${r}:58,58$ \\
                & & TT-cross & 226 & 0.24X & $1.1\!\times\!10^{-15}$ & ${r}:58,58$ \\
                & & rTTSVD & 23.4 & 2.3X & $4.8\!\times\!10^{-15}$ & ${r}:58,58$ \\
                & & FastTT & 1.67 & 32X & $2.4\!\times\!10^{-15}$ & $R\!:\!1920;~\tilde{{r}}\!:\!58,\!58;~{r}\!:\!58,\!58$ \\
                \midrule
                \multirow{4}*{30} & \multirow{4}*{random} & TT-SVD & 762 & -- & $3.4\!\times\!10^{-15}$ & ${r}:88,88$ \\
                & & TT-cross & 2725 & 0.28X & $6.2\!\times\!10^{-15}$ & ${r}:88,88$ \\
                & & rTTSVD & 67.0 & 11X & $4.2\!\times\!10^{-15}$ & ${r}:88,88$ \\
                & & FastTT & 12.4 & 61X & $3.3\!\times\!10^{-15}$ & $R\!:\!4380;~\tilde{{r}}\!:\!88,\!88;~{r}\!:\!88,\!88$ \\
                \midrule
                \multirow{4}*{40} & \multirow{4}*{random} & TT-SVD & NA & -- & NA & NA \\
                & & TT-cross & NA & -- & NA & NA \\
                & & rTTSVD & 597 & -- & $5.0\!\times\!10^{-15}$ & ${r}:118,118$ \\
                & & FastTT & 57.5 & -- & $2.6\!\times\!10^{-15}$ & \small{$R\!:\!7840;\tilde{{r}}\!:\!118,\!118;{r}\!:\!118,\!118$} \\
                \bottomrule
            \end{tabular}
            \begin{tablenotes}
                \footnotesize
                \item[1] $\varepsilon_{\text{actual}}=\frac{\|\ten{A}-\ten{B}\|_F}{\|\ten{A}\|_F}$. The same below.
                \item[2] $R$ is the number of nonzero $p$-fibers. $\tilde{r}$ is the TT-ranks after parallel-vector rounding. $r$ is the final TT-ranks.
            \end{tablenotes}
        \end{threeparttable}
    \end{table}
    
    As seen from \cref{tab:exp_fdm}, FastTT can convert large sparse matrices much faster than the TT-SVD with up to 240X speedup. These experiments also prove that the \textit{Depar} procedure can greatly reduce the TT-rank and thus simplify the computation of the TT-rounding procedure. Like in the first experiment, the TT-cross algorithm is faster when the TT-ranks are low but gets slower when the ranks grow. The results of rTTSVD are obtained by setting the TT-ranks same as those obtained by TT-SVD and FastTT. From the result we can see the rTTSVD algorithm is not as fast as FastTT even if we know the proper TT-ranks. 
    
    If we set $n=40$ with random coefficients, the TT-SVD/TT-cross algorithm cannot produce any result in a reasonable time, while FastTT finishes in 57.5 seconds with a resulting TT-rank of $r=118$.
    
    \subsection{Data of road network}
    
    A directed/undirected graph with $N$ nodes is equivalent to its adjacency matrix $\mat{A}\inR{N\times N}$, which can also be decomposed into an MPO if we properly factorize its order $N=n_1\times\cdots\times n_d$. This may benefit some data mining applications.
    In this experiment we use the undirected graph \textit{roadNet-PA\footnote{Road network of Pennsylvania. \url{http://snap.stanford.edu/data/roadNet-PA.html}}} from SNAP~\cite{snapnets}. Since the graph is fairly large, we only take the subgraph of the first $N$ nodes as our data and preprocess its adjacency matrix by performing  reverse Cuthill-McKee ordering~\cite{George:1981:CSL:578296}. Additionally, different desired accuracy tolerances $\varepsilon$ and the actual relative error are tested in this experiment. The truncation parameters in \cref{alg:tt_rounding} is either set as $\delta_k=\frac{\varepsilon}{\sqrt{p-1}+\sqrt{d-p}}\|\ten{A}\|_F$ or determined by \cref{alg:dynamic}. The results are shown in \cref{tab:exp_road}. The TT-cross/rTTSVD algorithm is not tested because both of them require the TT-ranks to be set the same, which is not possible in this experiment.
    
    \begin{table}[htbp]
        \centering
        \caption{Experimental results on converting the data of roadNet-PA.}
        \label{tab:exp_road}
        \begin{threeparttable}
            \begin{tabularx}{\textwidth}{@{ }@{\extracolsep{\fill}}ccccccc}
                \toprule
                $N$ & $\varepsilon$ & method\tnote{1} & time (s) & speedup & TT-ranks & $\varepsilon_{\text{actual}}$ \\
                \midrule
                \multirow{2}*{$20^3$} & \multirow{2}*{$1\times10^{-14}$} & TT-SVD & 75.4 & -- & 58, 400 & $3.7\times10^{-15}$ \\
                & & FastTT & 14.1 & 5.3X & 58, 400 & $3.3\times10^{-15}$ \\
                \midrule
                \multirow{3}*{$20^3$} & \multirow{3}*{$5\times10^{-1}$} & TT-SVD & 62.2 & -- & 31, 281 & $4.8\times10^{-1}$ \\
                & & FastTT & 11.8 & 5.3X & 31, 281 & $4.8\times10^{-1}$ \\
                & & FastTT$^+$ & 10.4 & 6.0X & 55, 209 & $5.0\times10^{-1}$ \\
                \midrule
                \multirow{2}*{$10^4$} & \multirow{2}*{$1\times10^{-14}$} & TT-SVD & 833 & -- & 28, 1407, 70 & $3.9\times10^{-15}$ \\
                & & FastTT & 23.3 & 34X & 28, 1407, 70 & $4.3\times10^{-15}$ \\
                \midrule
                \multirow{3}*{$10^4$} & \multirow{3}*{$1\times10^{-2}$} & TT-SVD & 839 & -- & 28, 1390, 70 & $5.5\times10^{-3}$ \\
                & & FastTT & 24.4 & 34X & 28, 1395, 70 & $3.8\times10^{-3}$ \\
                & & FastTT$^+$ & 24.2 & 35X & 28, 1377, 70 & $9.8\times10^{-3}$ \\
                \bottomrule
            \end{tabularx}
            \begin{tablenotes}
                \footnotesize
                \item[1] FastTT: use \cref{alg:tt_rounding} for TT-rounding;  FastTT$^{+}$: use \cref{alg:dynamic} for TT-rounding.
            \end{tablenotes}
        \end{threeparttable}
    \end{table}
    
    Again, for sparse graphs our FastTT algorithm is much faster than TT-SVD. Also, the actual relative errors are shown to be less than the given $\varepsilon$. If $\varepsilon$ is small enough, the TT-rank obtained by FastTT is the same as those obtained by TT-SVD. Otherwise, \cref{alg:dynamic} (used in FastTT$^+$) usually produces lower TT-ranks and a little bit higher relative error than \cref{alg:tt_rounding} (used in FastTT) which sets unified truncation parameters.
    
    \section{Conclusions}
    
    This paper analyzes several state-of-the-art algorithms for the computation of the TT decomposition and proposes a faster TT decomposition algorithm for sparse tensors. We prove the correctness and complexity of the algorithm and demonstrate the advantages and disadvantages of each algorithm.
    
    In the subsequent experiments, we have verified the actual performance of each algorithm and confirmed our theoretical analysis. The experimental results also show that our proposed FastTT algorithm for sparse tensors is indeed an algorithm with excellent efficiency and versatility. Previous state-of-the-art algorithms are mainly limited by the tensor size whereas our proposed algorithm is mainly limited by the number of non-zero elements. As a result, the TT decomposition can be computed quickly regardless of the number of dimensions.
    This algorithm therefore is very promising to tackle tensor applications that were previously unimaginable, just like the large-scale use of previous sparse matrix algorithms.
    
    \section{Acknowledgments}
    
    Funding: This work was supported by the National Natural Science Foundation of China [grant number 61872206].
    
    \appendix
    
    \section{Proof of \cref{thm:dyn}}\label{proof:dyn}
    
    \begin{proof}
        
        From the proof of \cref{thm:tt_rounding}, we know that
        \begin{align*} \|\ten{A}-\ten{B}\|_F\le\sqrt{\sum_{k=p}^{d-1}\|\mat{E}_k\|_F^2}+\sqrt{\sum_{k=2}^{p}\|\mat{E}_{k-1}\|_F^2}.
        \end{align*}
        
        Now we are going to use a loop invariant~\cite[pp. 18-19]{cormen2009introduction} to prove the correctness of \cref{alg:dynamic}. The loop invariant for Loop~\ref{line:loop1_start}-\ref{line:loop1_end} is
        \begin{align*}
            \delta_\text{right}^2+\sum_{i=p}^{k-1}\|\mat{E}_i\|_F^2= \frac{d-p}{\sqrt{d-p}+\sqrt{p-1}}\varepsilon^2\|\ten{A}\|_F^2.
        \end{align*}
        
        \textbf{Initialization}: Before the first iteration $k=p$, $\sum_{i=p}^{k-1}\|\mat{E}_i\|_F^2=0$ and $\delta_\text{right}=\frac{\sqrt{d-p}}{\sqrt{d-p}+\sqrt{p-1}}\varepsilon\|\ten{A}\|_F$. Thus the invariant is satisfied.
        
        \textbf{Maintenance}: After each iteration, $\delta_\text{right}^2$ is decreased by $\|\mat{E}_k\|_F^2$ and $\sum_{i=p}^{k-1}\|\mat{E}_i\|_F^2$ is increased by $\|\mat{E}_k\|_F^2$. Thus the invariant remains satisfied.
        
        \textbf{Termination}: When the loop terminates at $k=d$. Again the loop invariant is satisfied. This means that
        \begin{align*}
            \delta_\text{right}^2+\sum_{i=p}^{d-1}\|\mat{E}_i\|_F^2&= \frac{d-p}{\sqrt{d-p}+\sqrt{p-1}}\varepsilon^2\|\ten{A}\|_F^2\\
            \Rightarrow\quad\sqrt{\sum_{k=p}^{d-1}\|\mat{E}_k\|_F^2}&\le\frac{\sqrt{d-p}}{\sqrt{d-p}+\sqrt{p-1}}\varepsilon\|\ten{A}\|_F.
        \end{align*}
        
        Similarly we can prove that
        \begin{align*}
            \sqrt{\sum_{k=2}^{p}\|\mat{E}_{k-1}\|_F^2}\le\frac{\sqrt{p-1}}{\sqrt{d-p}+\sqrt{p-1}}\varepsilon\|\ten{A}\|_F,
        \end{align*}
        is satisfied after Loop~\ref{line:loop2_start}-\ref{line:loop2_end}.
        
        Thus
        \begin{align*} \|\ten{A}-\ten{B}\|_F&\le\sqrt{\sum_{k=p}^{d-1}\|\mat{E}_k\|_F^2}+\sqrt{\sum_{k=2}^{p}\|\mat{E}_{k-1}\|_F^2}.\\
            &\le\frac{\sqrt{d-p}}{\sqrt{d-p}+\sqrt{p-1}}\varepsilon\|\ten{A}\|_F+\frac{\sqrt{p-1}}{\sqrt{d-p}+\sqrt{p-1}}\varepsilon\|\ten{A}\|_F\\
            &=\varepsilon\|\ten{A}\|_F.
        \end{align*}
        
    \end{proof}

    \bibliographystyle{unsrt}
    \bibliography{refs}

\begin{thebibliography}{10}

\bibitem{hitchcock1927expression}
Frank~L Hitchcock.
\newblock The expression of a tensor or a polyadic as a sum of products.
\newblock {\em Studies in Applied Mathematics}, 6(1-4):164--189, 1927.

\bibitem{tucker1966some}
Ledyard~R Tucker.
\newblock Some mathematical notes on three-mode factor analysis.
\newblock {\em Psychometrika}, 31(3):279--311, 1966.

\bibitem{oseledets2011tensor}
Ivan~V Oseledets.
\newblock Tensor-train decomposition.
\newblock {\em SIAM Journal on Scientific Computing}, 33(5):2295--2317, 2011.

\bibitem{oseledets2012solution}
Ivan~V Oseledets and SV~Dolgov.
\newblock Solution of linear systems and matrix inversion in the {TT}-format.
\newblock {\em SIAM Journal on Scientific Computing}, 34(5):A2718--A2739, 2012.

\bibitem{zhang2014enabling}
Zheng Zhang, Xiu Yang, Ivan~V Oseledets, George~E Karniadakis, and Luca Daniel.
\newblock Enabling high-dimensional hierarchical uncertainty quantification by
  {ANOVA} and tensor-train decomposition.
\newblock {\em IEEE Transactions on Computer-Aided Design of Integrated
  Circuits and Systems}, 34(1):63--76, 2014.

\bibitem{liu2015model}
Haotian Liu, Luca Daniel, and Ngai Wong.
\newblock Model reduction and simulation of nonlinear circuits via tensor
  decomposition.
\newblock {\em IEEE Transactions on Computer-Aided Design of Integrated
  Circuits and Systems}, 34(7):1059--1069, 2015.

\bibitem{zhang2017tensor}
Zheng Zhang, Kim Batselier, Haotian Liu, Luca Daniel, and Ngai Wong.
\newblock Tensor computation: A new framework for high-dimensional problems in
  {EDA}.
\newblock {\em IEEE Transactions on Computer-Aided Design of Integrated
  Circuits and Systems}, 36(4):521--536, 2017.

\bibitem{batselier2017tensor}
Kim Batselier, Zhongming Chen, and Ngai Wong.
\newblock {Tensor network alternating linear scheme for MIMO Volterra system
  identification}.
\newblock {\em Automatica}, 84:26--35, 2017.

\bibitem{oseledets2010approximation}
Ivan~V Oseledets.
\newblock Approximation of $2^d\times2^d$ matrices using tensor decomposition.
\newblock {\em SIAM Journal on Matrix Analysis and Applications},
  31(4):2130--2145, 2010.

\bibitem{kressner2016low}
Daniel Kressner and Andr{\'e} Uschmajew.
\newblock On low-rank approximability of solutions to high-dimensional operator
  equations and eigenvalue problems.
\newblock {\em Linear Algebra and its Applications}, 493:556--572, 2016.

\bibitem{batselier2018computing}
Kim Batselier, Wenjian Yu, Luca Daniel, and Ngai Wong.
\newblock Computing low-rank approximations of large-scale matrices with the
  tensor network randomized {SVD}.
\newblock {\em SIAM Journal on Matrix Analysis and Applications},
  39(3):1221--1244, 2018.

\bibitem{wang2017efficient}
Wenqi Wang, Vaneet Aggarwal, and Shuchin Aeron.
\newblock Efficient low rank tensor ring completion.
\newblock In {\em IEEE International Conference on Computer Vision (ICCV)},
  pages 5698--5706, 2017.

\bibitem{ko2018fast}
Ching-Yun Ko, Kim Batselier, Wenjian Yu, and Ngai Wong.
\newblock Fast and accurate tensor completion with total variation regularized
  tensor trains.
\newblock {\em arXiv preprint arXiv:1804.06128}, 2018.

\bibitem{wang2018tensor}
Wenqi Wang, Vaneet Aggarwal, and Shuchin Aeron.
\newblock Tensor train neighborhood preserving embedding.
\newblock {\em IEEE Transactions on Signal Processing}, 66(10):2724--2732,
  2018.

\bibitem{wang2017support}
Yongkang Wang, Weicheng Zhang, Zhuliang Yu, Zhenghui Gu, Hao Liu, Zhaoquan Cai,
  Congjun Wang, and Shihan Gao.
\newblock Support vector machine based on low-rank tensor train decomposition
  for big data applications.
\newblock In {\em 2017 12th IEEE Conference on Industrial Electronics and
  Applications (ICIEA)}, pages 850--853. IEEE, 2017.

\bibitem{chen2017parallelized}
Zhongming Chen, Kim Batselier, Johan~AK Suykens, and Ngai Wong.
\newblock Parallelized tensor train learning of polynomial classifiers.
\newblock {\em IEEE transactions on neural networks and learning systems},
  29(10):4621--4632, 2017.

\bibitem{xu2018whole}
Xiaowen Xu, Qiang Wu, Shuo Wang, Ju~Liu, Jiande Sun, and Andrzej Cichocki.
\newblock Whole brain {fMRI} pattern analysis based on tensor neural network.
\newblock {\em IEEE Access}, 6:29297--29305, 2018.

\bibitem{oseledets2010tt}
Ivan Oseledets and Eugene Tyrtyshnikov.
\newblock {TT-cross approximation for multidimensional arrays}.
\newblock {\em Linear Algebra and its Applications}, 432(1):70--88, 2010.

\bibitem{savostyanov2011fast}
Dmitry Savostyanov and Ivan Oseledets.
\newblock Fast adaptive interpolation of multi-dimensional arrays in tensor
  train format.
\newblock In {\em The 2011 International Workshop on Multidimensional (nD)
  Systems}, pages 1--8. IEEE, 2011.

\bibitem{huber2017randomized}
Benjamin Huber, Reinhold Schneider, and Sebastian Wolf.
\newblock A randomized tensor train singular value decomposition.
\newblock In {\em Compressed Sensing and its Applications}, pages 261--290.
  Springer, Cham, 2017.

\bibitem{halko2011finding}
Nathan Halko, Per-Gunnar Martinsson, and Joel~A Tropp.
\newblock Finding structure with randomness: Probabilistic algorithms for
  constructing approximate matrix decompositions.
\newblock {\em SIAM review}, 53(2):217--288, 2011.

\bibitem{goreinov1997theory}
Sergei~A Goreinov, Eugene~E Tyrtyshnikov, and Nickolai~L Zamarashkin.
\newblock A theory of pseudoskeleton approximations.
\newblock {\em Linear algebra and its applications}, 261(1-3):1--21, 1997.

\bibitem{hubig2017generic}
C~Hubig, IP~McCulloch, and U~Schollw{\"o}ck.
\newblock Generic construction of efficient matrix product operators.
\newblock {\em Physical Review B}, 95(3):035129, 2017.

\bibitem{xerus}
Benjamin Huber and Sebastian Wolf.
\newblock Xerus - {A} general purpose tensor library, 2014--2017.

\bibitem{dolgov2015corrected}
Sergey~V Dolgov and Dmitry~V Savostyanov.
\newblock Corrected one-site density matrix renormalization group and
  alternating minimal energy algorithm.
\newblock In {\em Numerical Mathematics and Advanced Applications-ENUMATH
  2013}, pages 335--343. Springer, 2015.

\bibitem{snapnets}
Jure Leskovec and Andrej Krevl.
\newblock {SNAP Datasets}: {Stanford} large network dataset collection, June
  2015.

\bibitem{George:1981:CSL:578296}
Alan George and Joseph~W. Liu.
\newblock {\em Computer Solution of Large Sparse Positive Definite}.
\newblock Prentice Hall Professional Technical Reference, 1981.

\bibitem{cormen2009introduction}
Thomas~H Cormen, Charles~E Leiserson, Ronald~L Rivest, and Clifford Stein.
\newblock {\em Introduction to Algorithms}.
\newblock MIT press, 2009.

\end{thebibliography}
    
\end{document}